\documentclass[twoside,11pt]{article}

\usepackage{jmlr2e}
\usepackage{lastpage}
\usepackage{booktabs}
\usepackage{amsmath}
\newtheorem{assumption}{Assumption}

\begin{document}

\title{Towards  regularized learning from functional data with covariate shift}

\author{
\name Markus Holzleitner \email holzleitner@ml.jku.at \\
\addr Institute for Machine Learning\\
Johannes Kepler University Linz\\
Austria
\AND
\name Sergiy Pereverzyev Jr \email sergiy.pereverzyev@i-med.ac.at \\
\addr Department of Neuroradiology\\
Medical University of Innsbruck\\
Austria
\AND
\name Sergei V. Pereverzyev \email sergei.pereverzyev@oeaw.ac.at \\
\addr Johann Radon Institute for Computational and Applied Mathematics\\
Austrian Academy of Sciences\\
Austria
\AND
\name Vaibhav Silmana \email vibhusilmana@gmail.com \\
\addr Department of Mathematics\\
Indian Institute of Technology Delhi\\
India
\AND
\name S. Sivananthan \email siva@maths.iitd.ac.in \\
\addr Department of Mathematics\\
Indian Institute of Technology Delhi\\
India
}

\editor{...}

\maketitle

\begin{abstract}
This paper investigates a general regularization framework for unsupervised domain adaptation in vector-valued regression under the covariate shift assumption, utilizing vector-valued reproducing kernel Hilbert spaces (vRKHS).
Covariate shift occurs when the input distributions of the training and test data differ, introducing significant challenges for reliable learning. By restricting our class of hypothesis space, we develop a practical operator-learning algorithm capable of handling functional outputs. We establish optimal convergence rates for the proposed framework under a general source condition, providing a theoretical foundation for regularized learning in this setting. {We also propose an aggregation-based approach that forms a linear combination of estimators corresponding to different regularization parameters and different kernels. The proposed approach addresses the challenge of selecting appropriate tuning parameters that is crucial for constructing a good estimator, and we provide a theoretical justification for its effectiveness.} Further, we illustrate our proposed method 
 on real-world face image dataset, demonstrating the method’s robustness and effectiveness in mitigating distributional discrepancies under covariate shift. 
\end{abstract}


\section{Introduction}
In statistical learning theory, the fundamental goal is to learn a relationship between input and output variables based on observed data. 
Formally, given an input space $X$ and an output space $Y$, we assume that the dataset $\mathbf{z} =$ $\{(x_i, y_i)\}_{i=1}^n \subset X \times Y$, is drawn independently according to an unknown probability distribution $\rho$ on $X \times Y$. 
{The aim of the learning algorithm is to construct a function $f:X\rightarrow Y$ that approximates the underlying dependency between the input $x$ and the output $y$, achieving a small prediction error with respect to $\rho$.}
However, in many real-world situations, the training and testing data are not sampled from the same distribution; 
this leads to the problem of \emph{domain adaptation}. 

Let us denote the source distribution as  $p(x,y)$ and the target distribution as $q(x,y)$ on $X\times Y$. If the source and target distributions are entirely unrelated, then it is impossible to minimize the target risk based solely on source samples \cite{impossibility}.
Hence, to make domain adaptation feasible, one typically assumes a structural
relationship between the two distributions. A particularly important and tractable case is the
\emph{covariate shift} setting. 

\vspace{0.4em}

Generally, covariate shift refers to a change in the input distribution  compared to the training data used to build learning algorithms. In case of finite dimensional inputs and outputs, this issue has been extensively discussed in the literature, starting with the seminal work by  \cite{shimodaira2000improving}. To the best of our knowledge, the covariate shift issue is much less studied in the case of functional data, when both the output and the input variables are functions. At the same time, functional data with covariate shift frequently appear in various practical applications where changes in input distributions are caused by environmental and operational factors, as it is the case, for example, in Structural health monitoring discussed in the study by  \cite{Wittenberg_Neumann_Mendler_Gertheiss_2025}. In that study, the authors propose an adjustment to covariate-induced variations by training a functional mixed additive model for function-on-function regression,  in which the regression function is assumed to be expanded in a finite sum of chosen basis functions such as  B-splines, for example. Note that such family of models is covered by class of learning algorithms in vector-valued reproducing kernel Hilbert spaces (vRKHS) that has been recently investigated in the study by  \cite{meunier2024optimal}. 

At the same time, we are not aware of any results on learning in vRKHS from functional data with covariate shift, and this is the goal of the present study to shed light on the possibility of such learning. For this aim, we combine the technique developed in  \cite{meunier2024optimal} with the ideas presented in \cite{Profsergie2022}.

Mathematically, the covariate shift assumption posits the following relationship between the source and target distributions:
\[
 p(x,y) = p(y|x) \, p_X(x), \quad
    q(x,y) = p(y | x) \, q_X(x),
\]
where \( p_X(x) \) and \( q_X(x) \) are the marginal distributions and \( p(y|x) \) is the same conditional distribution  that is related by the expression $dp(y| x) = \mathcal{K}(dy, x)$ to the so-called Markov kernel $\mathcal{K}(dy, x)$, which has been employed in the analysis \cite{meunier2024optimal}.

Under this setting, one typically has access to labeled samples \(\{(x_i, y_i)\}_{i=1}^n\) from the source distribution \(p(x,y)\) and unlabeled inputs \(\{x'_j\}_{j=1}^m\) from the target marginal distribution \(q_X\). The goal is to minimize the expected risk over the target domain $(X\times Y, q)$:
\begin{equation}\label{eq:targetrisk}
\mathcal{R}_q(f) 
= \int_{X \times Y} \big\|f(x)-y\big\|_Y^2 \, dq(x,y).
\end{equation}

\noindent
Assuming
\[
    \int_{X \times Y} \|y\|_Y^2 \, dq(x,y) < \infty,
\]
it can be shown that the minimum of the expected risk $\mathcal{R}_q(f)$ over  Bochner space $L_2(q_X,Y)$ of square-integrable functions
$f : X \to Y$ with respect to the marginal distribution $q_X$ on $X$ is attained at the \emph{regression function}
\[
    f_q(x) = \int_Y y \, dp(y| x).
\]
\vspace{0.3em}

Since the conditional distribution is identical,
the  regression functions corresponding to the source and target domains coincide, i.e., $f_p = f_q$.
We denote this common regression function by $f^*$.
Hence, given the available data, the goal is to minimize the target risk
over $L_2(q_X,Y)$ by constructing an estimator $f_\mathbf{z}$ that closely approximates
the regression function $f^*$, with the key distinction that, unlike standard
supervised learning where the minimization is performed in $L_2(p_X,Y)$ under
the source marginal $p_X$, domain adaptation seeks to minimize the error in
$L_2(q_X,Y)$ determined by the target marginal $q_X$. The quality of this approximation is typically quantified in
terms of the \emph{excess risk}, $\mathcal{R}_q(f_\mathbf{z}) - \mathcal{R}_q(f^*)$. From \cite{devito(2007)} we know that it can be expressed as
\[
\mathcal{R}_q(f_\mathbf{z}) - \mathcal{R}_q(f^*)
= \big\|f_\mathbf{z} - f^*\big\|_{L_2(q_X,Y)}^2 \; .
\]

\noindent

\vspace{0.5em}

Covariate shift has been studied in finite-dimensional and
scalar-output settings.
Recent advances have extended this analysis to more general setups with improved convergence rates
(see, ~\cite{DBLP:conf/iclr/DinuHBNHEMPHZ23}, \cite{Profsergie2022}, \cite{shimodaira2000improving}).
In \cite{Profsergie2022}, convergence rates were established under general source condition characterized by splitting form of the index function. Moreover, a number of recent studies have investigated methods for addressing the case of unbounded density ratios; see \cite{gogolashvili2023importance, ma2023optimally, fan2025spectral, liu2025spectral}. Furthermore, convergence analyses have been established for
computationally efficient methods such as Nyström subsampling~\cite{covnystrom}, which substantially reduce
the computational complexity of kernel-based algorithms while preserving optimal statistical guarantees.  

\vspace{0.4em}

The vRKHS framework offers an approach to regression problems involving functional or infinite-dimensional outputs (see, \cite{VRKHS_REF}, \cite{rastogi2017optimal}, \cite{meunier2024optimal}). Notably,~\cite{rastogi2017optimal} derived optimal convergence rates for Tikhonov and general regularization schemes under general source conditions for trace-class kernels, thereby extending the classical scalar-valued results to the vector-valued framework. More recently,~\cite{meunier2024optimal} provided a comprehensive theoretical framework for
vector-valued regularization methods, encompassing both well-specified and misspecified cases.
Within this framework, optimal convergence rates were derived for the general multiplicative kernel under interpolation space conditions in the $\gamma$-norm. Together, these results highlight that  vRKHS provides a principled framework for modern learning problems such as multi-task regression, functional data analysis, where outputs are inherently in high or infinite-dimensional space. 

Motivated by this, we first propose an operator learning algorithm and then establish convergence rates under a general source condition for a broad class of regularization schemes in vector-valued regression under covariate shift.

It is also important to note that the performance of regularized learning with reproducing kernels crucially depends on how regularization parameters and kernels are chosen. On the other hand, to the best of our knowledge, the issue with the above choices has not been raised yet in the context of learning from functional data. In this paper, we are going to shed light on this problem.

\vspace{0.5em}

\noindent
\textbf{Main contributions.} In this paper, we address the problem of \emph{vector-valued regression under covariate shift}. Our contributions are as follows: \begin{itemize}

\item We propose an operator-learning algorithm for vector-valued regression in covariate shift adaptation under the framework of vRKHS. 

\item  We establish optimal convergence rates for the above framework in the \( L_2(q_X, Y) \) norm and vRKHS norm for general regularization under general source condition.

\item We propose an aggregation strategy that alleviates the difficulty of tuning the regularization and kernel parameters, thereby enhancing the practicality of the method in high-dimensional settings. 
\end{itemize} These results mark the first systematic treatment of vector-valued covariate-shift regression, bridging theoretical development with real-world practice. 

\medskip 

\noindent
\textbf{Organization of the paper.} The remainder of this paper is organized as follows. In section~\ref{sec:preliminaries}, we introduce the necessary preliminaries and notation required for our analysis. 
Section~\ref{Construction} presents the construction of our proposed algorithm for operator learning under covariate shift. In section \ref{sec:proof}, we provide learning rates for general regularization under general source condition. In section~\ref{Aggregation}, we construct the aggregation approach and establish its convergence guarantees. Section~\ref{numerical experiments} reports numerical experiments on a real world face dataset, illustrating the practical effectiveness of our method. The proofs of the auxiliary lemmas are deferred to Section~\ref{Auxillary_results}. 

\section{Preliminaries}\label{sec:preliminaries}

 We begin by introducing background material on vRKHS, which constitutes the hypothesis class used in our framework. Detailed expositions can be found in \cite{VRKHS(BASIC)} and \cite{carmeli2008vectorvaluedreproducingkernel}. Throughout this article, we assume that $X$ is a second-countable, locally compact Hausdorff space, and that $Y$ is an infinite-dimensional real separable Hilbert space $(Y,\langle \cdot,\cdot \rangle)_Y$.

\begin{definition}
Let $X \neq \emptyset $ and $(Y, \langle \cdot, \cdot \rangle_Y)$ be a real Hilbert space. 
A Hilbert space $(\mathcal{G}, \langle \cdot, \cdot \rangle_{\mathcal{G}}) \subset Y^X$ 
is called a {reproducing kernel Hilbert space} if, for every $x \in X$ and $y \in Y$, 
the mapping 
\[
f \mapsto \langle y, f(x) \rangle_Y
\]
is a continuous linear functional on $\mathcal{G}$.

\end{definition}

\noindent
By virtue of Riesz representation theorem, for each $x \in X$ and $y \in Y$ there 
exists a linear operator $K_x : Y \to \mathcal{G}$ such that
\begin{center}
  $\langle y, f(x) \rangle_Y = \langle K_x y, f \rangle_{\mathcal{G}}, 
\qquad \forall f \in \mathcal{G}.
$    
\end{center}

\vspace{0.4em}

\noindent
As a result, the corresponding adjoint operator $K_x^* : \mathcal{G} \to Y$ 
is defined as
\begin{center}
$K_x^* f = f(x).$

\end{center}
\vspace{0.5em}
\noindent
Through the operator $K_x$, we define the \emph{operator-valued kernel}
$
K : X \times X \to {L}(Y)$
by
\[
  K(x,t)y := (K_t y)(x), \quad \forall x,t \in X,\; y \in Y.
\]

\noindent
where $L(Y)$ denotes the space of bounded linear operators from $Y$ to $Y$.

\vspace{0.2em}

\noindent
The function $K$ is called the \emph{reproducing kernel} associated with the space 
$\mathcal{G}$. It satisfies the following properties:
\begin{enumerate}
    \item $K(x, t) = K(t, x)^*, \quad \forall \, x, t \in X.$
    \item For any $m \in \mathbb{N}$, $\{x_i\}_{i=1}^m \subset X$, and $\{y_i\}_{i=1}^m \subset Y$, we have
    \[
        \sum_{i,j=1}^{m} \langle y_i, K(x_i, x_j) y_j \rangle_Y \ge 0.
    \]
\end{enumerate}

\noindent
Conversely, every operator-valued function 
$K : X \times X \to {L}(Y)$ satisfying (1) and (2) uniquely determines a 
vRKHS $\mathcal{G}$. Hence there exists a one-to-one correspondence between vector-valued kernels and vRKHS. For further details, we refer the reader to~\cite{one-one}. 

\vspace{0.5em}
\noindent

As mentioned earlier, the target labels are not directly accessible. To address this challenge, one of the most widely used approaches for approximating the minimizer $f^*$ of the target expected risk $\mathcal{R}_q(f)$, based on data $\mathbf{z} = \{(x_i, y_i)\}_{i=1}^n$ sampled from the source distribution $p(x,y)$, is penalized least squares regression combined with sample reweighting. This method is commonly known as importance weighted regularized least square (IWRLS).

\vspace{0.4em}
In IWRLS, it is assumed 
that there exists a  reweighting function  
\begin{center}
    
$\beta : {X} \to \mathbb{R}_+, 
\qquad 
dq_X(x) = \beta(x)\, dp_X(x),
$
\end{center}

\noindent
which relates the target and source input distributions.
Using the relationship between $q_X$ and $p_X$, the expected risk over target distribution (\ref{eq:targetrisk}) can be equivalently 
expressed in terms of the source distribution as

\[
\mathcal{R}_q(f) 
= \int_{X \times Y} \big\|f(x)-y\big\|_Y^2 \, \beta(x)\, dp(x,y).
\]

\medskip
\noindent
The empirical risk over the target distribution, expressed in terms of the source dataset \textbf{z}, is given by
\begin{equation} \label{eq:empirical-risk-beta}
\mathcal R_\mathbf{z}(f) 
= \frac{1}{n} \sum_{i=1}^n \beta(x_i)\, \big\| f(x_i) - y_i \big\|_{Y}^2.
\end{equation}

\noindent
To rewrite (\ref{eq:empirical-risk-beta}) in operator form, 
we define the sampling operator $S_{X_S}$ on the source dataset as

\begin{center}
    
$S_{X_S} : \mathcal{G} \longrightarrow Y^n, 
\qquad 
S_{X_S}(f) = \left(f(x_1), f(x_2), \ldots, f(x_n)\right),$
\end{center}

\noindent
and its adjoint 
\[
S_{X_S}^* : Y^n \longrightarrow \mathcal{G}, 
\qquad 
S_{X_S}^*(\mathbf{y}) = \frac{1}{n} \sum_{i=1}^n  K_{x_i}y_i  ,
\]

\noindent
where $\mathbf{y} = ( y_1,y_2,\ldots, y_n ).$

\vspace{0.3em}

Furthermore, let  

\[
\mathcal B^{1/2}
= \mathrm{diag}\left(\sqrt{\beta(x_1)}, \dots, \sqrt{\beta(x_n)}\right)
\]
denote the diagonal matrix whose entries are the square roots of the importance weights $\beta(x_i)$.

\vspace{0.6em}

\noindent
Using sampling operator, we may express the empirical risk (\ref{eq:empirical-risk-beta}) as
\begin{equation} \label{eq:empirical-risk-operator}
    \mathcal{R}_\mathbf{z}(f) \;=\; \left\|\mathcal B^{1/2} S_{X_S} f \;-\; \mathcal B^{1/2} \mathbf{y} \right\|_n^2,
\end{equation}
where the norm $\|\cdot\|_n$ is given by

\[
\| \mathbf{y} \|_n^2 = \frac{1}{n} \sum_{i=1}^n \| y_i \|_Y^2.
\]

\noindent
Taking  Fréchet derivative of (\ref{eq:empirical-risk-operator}) and setting it to zero, we obtain the following equation for the minimizer of the empirical risk

\begin{equation}\label{Sxseq}
       \left(S_{X_S}^* \mathcal B S_{X_S}\right) f \;=\;  \, S_{X_S}^* \mathcal B \mathbf{y}.
\end{equation}

\begin{remark}
  Note that, the weights appearing in ~(\ref{Sxseq}) are assumed to be {exact}, 
i.e., they are directly available and not approximated. 
In practice, however, such weights are typically unknown and need to be estimated.
An effective algorithm for estimating the reweighting function~$\beta$, known as 
KuLSIF (Kernel unconstrained least-squares importance fitting), 
was proposed in \cite{IWRLS1}. 
For the sake of completeness, we briefly introduce this algorithm in 
Section~\ref{Construction}.

\end{remark}

In practical situations, data is often noisy or inaccurate; therefore, solving equation (\ref{Sxseq}) directly may produce a solution that excessively conforms to the training data, thereby capturing noise rather than the underlying structure. Regularization mitigates this issue by penalizing overly complex functions, resulting in a smoother and more generalizable solution. Next, we define the \emph{regularization family} \cite{regularization}, which can be seen as a family of functions 
approximating the inversion map.

\begin{definition}\label{def:regularization}
Let $g_\lambda: [0,s] \to \mathbb{R}$, $0 < \lambda \leq s$, denote a family of functions. 
The family $\{g_\lambda\}_{\lambda > 0}$ is referred to as a \emph{regularization family} if it fulfills the following conditions:
\begin{itemize}
    \item[\textnormal{(i)}] There exists $D > 0$ such that 
    \[
        \sup_{\sigma \in (0,s]} |\sigma g_\lambda(\sigma)| \leq D.
    \]
    \item[\textnormal{(ii)}] There exists $B > 0$ such that 
    \[
        \sup_{\sigma \in (0,s]} |g_\lambda(\sigma)| \leq \frac{B}{\lambda}.
    \]
\end{itemize}
Furthermore, let $\gamma > 0$ be such that 
\[
    \sup_{\sigma \in (0,s]} |1 - \sigma g_\lambda(\sigma)| \leq \gamma.
\]

\noindent
The largest $\nu > 0$ for which \[ \sup_{\sigma \in (0,s]} |1 - \sigma g_\lambda(\sigma)| \, \sigma^\nu \leq \gamma_\nu \lambda^\nu \] holds is called the \emph{qualification} of the regularization family $g_\lambda$, where $\gamma_\nu$ is independent of $\lambda$.

\end{definition}
\begin{definition}\label{qualfication_covering}
Let $\nu>0$ be the qualification of a regularization family.  
We say that the qualification $\nu$ \emph{covers} the index function $\varphi$, if there exists a constant $c>0$ such that
\begin{equation}\label{eq:cover}
c\,\frac{t^{\nu}}{\varphi(t)}
\;\le\;
\inf_{\,t \le \sigma \le s}
\frac{\sigma^{\nu}}{\varphi(\sigma)},
\end{equation}
for all $0 < t \le s$.
Here, $\varphi$ is an index function, i.e., $\varphi:[0,\infty)\to[0,\infty)$ is continuous, strictly increasing, and satisfies $\varphi(0)=0$.
\end{definition}

\noindent
Next we give few examples of some standard regularization schemes.

\begin{enumerate}
    \item \textbf{Tikhonov Regularization:}
    \[
        g_\lambda(\sigma) = \frac{1}{\sigma + \lambda}.
    \]
    This is the classical and most widely used choice. It has qualification $\nu = 1$.

    \item \textbf{Iterated Tikhonov Regularization:}
    \[
        g_\lambda(\sigma) = \frac{1}{\sigma}
        \left( 1 - \left( \frac{\lambda}{\sigma + \lambda} \right)^{m} \right),
        \quad m \in \mathbb{N}.
    \]
    The parameter $m$ controls the number of iterations, and the qualification increases 
    linearly with $m$.

    \item \textbf{Spectral Cut-off Regularization:}
   \[
        g_\lambda(\sigma) = 
        \begin{cases}
            \dfrac{1}{\sigma}, & \text{if } \sigma \ge \lambda,\\[0.4em]
            0, & \text{otherwise.}
        \end{cases}
    \]

    This method effectively discards small singular values below $\lambda$, and 
    has infinite qualification.
\end{enumerate}

\noindent

With the use of a regularization family one can construct a regularized approximation $\widehat{f}_{\mathbf{z},\lambda}$ of the solution of the equation (\ref{Sxseq}) for the
minimizer of the empirical risk as follows

\begin{equation}\label{equation:genral f_Z}
\widehat{f}_{\mathbf{z},\lambda} \;=\; g_\lambda\left(S_{X_S}^* \mathcal B S_{X_S}\right) \, S_{X_S}^* \mathcal B \mathbf{y}.
\end{equation}

\section{Construction of  Operator Learning Algorithm for Covariate shift}\label{Construction}

\noindent

\noindent
Throughout this paper, we will  consider vRKHS induced by kernels of the form
\begin{equation}\label{multk}
K(x,t) \;=\; k(x,t)\, Id_Y,    
\end{equation}
(such kernels are referred as scalar multiplicative kernel) where $Id_Y : Y \to Y$ denotes the identity operator on Y.  
The scalar valued RKHS induced by kernel $k$ is denoted by $\mathcal{H}$. Next, we state a theorem that establishes an isomorphism between the vRKHS $\mathcal{G}$ and the Hilbert space of Hilbert--Schmidt operators from $\mathcal{H}$ to $Y$, denoted by $\mathcal{S}_2(\mathcal{H},Y)$. We use the notation $\mathcal{S}_2(\mathcal{H})$ to denote the space of Hilbert--Schmidt operators from $\mathcal{H}$ to $\mathcal{H}$.

\begin{theorem}[\cite{VRKHS_REF}, Theorem 1] Let $\phi:X\rightarrow \mathcal{H}$ denotes the feature map such that $\phi(x) = k(x,\cdot)$. Then for every function $f \in \mathcal{G}$ there exists a unique operator 
$C \in \mathcal S_2(\mathcal{H}, {Y})$ such that 

\begin{center}
    $f(\cdot) = C \phi(\cdot) $
\end{center}

\noindent
with 

\begin{center}
    $\|C\|_{\mathcal S_2(\mathcal{H}, {Y})} = \|f\|_{\mathcal{G}}$
\end{center}

\noindent
and vice versa. Hence $\mathcal{G} \simeq \mathcal S_2(\mathcal{H}, {Y})$ and
it follows that $\mathcal{G}$ can be written as

\begin{center}
    $\mathcal{G} = \{ f : {X} \to {Y} \mid 
    f = C \phi(\cdot), \; C \in \mathcal S_2(\mathcal{H}, {Y}) \}.$
\end{center}
\end{theorem}

\noindent
By applying above theorem to the empirical risk functional 
in~(\ref{eq:empirical-risk-beta}), we get 
\begin{equation}\label{equation:emphiirical_Riskin C}
\mathcal{R}_\mathbf{z}(C) 
    \;:=\; \frac{1}{n}\sum_{i=1}^n 
    \left\| y_i - C \phi(x_i) \right\|_{{Y}}^2 \, \beta(x_i) \, .
\end{equation}    

\noindent
Note that the minimization problem over $\mathcal{G}$ can be equivalently 
expressed as an minimization over the space of Hilbert--Schmidt operators, i.e.,
\vspace{0.3em}
\[
\min_{f \in \mathcal{G}} \; \mathcal{R}_{\mathbf z}(f)
\;=\;
\min_{C \in \mathcal{S}_2(\mathcal{H},Y)} \; \mathcal{R}_{\mathbf z}(C).
\]

By recasting the problem in the space of Hilbert--Schmidt operators, 
we effectively transform the search for an optimal function 
$f \in \mathcal{G}$ into the task of identifying the best operator 
$C \in \mathcal S_2(\mathcal{H}, {Y})$ that minimizes the same risk functional. Let $C_* \in \mathcal S_2(\mathcal{H},Y)$ denotes the minimizer of $\mathcal{R}_{\mathbf z}(C)$ over $\mathcal S_2(\mathcal{H},Y)$, then the minimizer over $\mathcal{G}$ (i.e., the regression function) is given by
\[
f^\ast(\cdot)=C_* \phi(\cdot).
\]
  
\noindent
Before stating the next proposition, we recall the notion of the tensor product of Hilbert spaces, 
which will be used in the operator formulation below. 
For a detailed exposition, see \cite{aubin}.

\medskip

Let $\mathcal H \otimes \mathcal{H'}$ denote the tensor product of the Hilbert spaces $\mathcal H$ and $\mathcal H'$.
 For $x \in \mathcal H$ and $x' \in \mathcal H'$, the elementary tensor $x \otimes x'$ is defined by
\[
x \otimes x' : \mathcal H' \to \mathcal H, 
\qquad 
y' \mapsto \langle y', x' \rangle_{\mathcal H'}\, x, 
\quad y' \in \mathcal H'.
\]
Since the tensor product space $\mathcal H \otimes \mathcal H'$ is isometrically isomorphic to the Hilbert space 
$\mathcal{S}_2(\mathcal H', \mathcal H)$, 
we do not distinguish between these two spaces and use them interchangeably.

\medskip

\begin{proposition}    

The minimizer of the empirical risk (\ref{equation:emphiirical_Riskin C}) satisfies the operator equation
\begin{center}
    $\widehat{C}_{YX}^\beta = C \, \widehat{C}_{X}^\beta,$
\end{center}
where
\[
\widehat{C}_{YX}^\beta = \frac{1}{n}\sum_{i=1}^n \beta(x_i) \, y_i \otimes \phi(x_i),
    \qquad
\widehat{C}_{X}^\beta = \frac{1}{n}\sum_{i=1}^n \beta(x_i) \, \phi(x_i) \otimes \phi(x_i).
\]
\end{proposition}
\noindent

\noindent
\begin{proof}
Let $C,H \in \mathcal S_2(\mathcal{H},Y)$.
Consider
\[
\mathcal{R}_{\mathbf z}(C+H) - \mathcal{R}_{\mathbf z}(C)
=
\frac{1}{n}\sum_{i=1}^n
\Bigl[
\bigl\| \beta^{1/2}(x_i)\bigl(y_i - (C+H)\phi(x_i)\bigr) \bigr\|_{Y}^2
-
\bigl\| \beta^{1/2}(x_i)\bigl(y_i - C\phi(x_i)\bigr) \bigr\|_{Y}^2
\Bigr].
\]

\noindent
Expanding the norm gives

\vspace{0.5em}
\noindent
\[
\mathcal{R}_\mathbf{z}(C+H) - \mathcal{R}_\mathbf{z}(C)
    = \frac{2}{n}\sum_{i=1}^n 
        \Big\langle \beta^{1/2}(x_i)\,(y_i - C \phi(x_i)),\,
        \beta^{1/2}(x_i)\, H \phi(x_i) \Big\rangle_{Y}
        + \mathcal{O}(\|H\|^2).
\]
\vspace{0.5em}

\noindent
Using the identity
\[
    \langle a,\, H b \rangle_{Y} 
    = \langle H,\, a \otimes b \rangle_{\mathcal S_2(\mathcal{H},Y)},
\]
we can rewrite the above equation as
\[
\mathcal{R}_{\mathbf z}(C+H) - \mathcal{R}_{\mathbf z}(C)
=
\frac{2}{n}\sum_{i=1}^n
\left\langle
H,\,
\beta(x_i)\, y_i \otimes \phi(x_i)
-
\beta(x_i)\, C\phi(x_i) \otimes \phi(x_i)
\right\rangle_{\mathcal{S}_2(\mathcal{H},Y)}
+\;
\mathcal O(\|H\|^2).
\]

\noindent
Hence, taking the Fréchet derivative of $\mathcal{R}_\mathbf{z}(C)$ and setting it equal to zero, we obtain
\begin{center}
    $\widehat{C}_{YX}^\beta = C \, \widehat{C}_{X}^\beta.$
\end{center}
\end{proof}
\noindent
Finally via regularization, the minimizer $C_*$ of the empirical risk (\ref{equation:emphiirical_Riskin C}) can be approximated by 

\begin{equation}\label{c-equation}
\widehat{C}_\lambda ^\beta
= \widehat{C}_{YX}^\beta \, g_\lambda \!\left(\widehat{C}_{X}^\beta\right).
\end{equation}


\noindent
The proof of the next proposition is obtained by adapting the analysis of Proposition~1 in \cite{meunier2024optimal} to our framework.

\begin{proposition}[Representer theorem for spectral filter with importance weights]
Let $\mathbf{K} = (k(x_i,x_j))_{i,j=1}^n$ denote the Gram matrix associated with the scalar kernel $k$, 
and let
\[
\mathcal B^{1/2}
=
\mathrm{diag}\left(\sqrt{\beta(x_1)}, \dots, \sqrt{\beta(x_n)}\right).
\]
{Define $k_{x,n} = (k(x,x_1), \ldots, k(x,x_n))^T \in \mathbb{R}^n$.} 
Then the approximation
(\ref{equation:genral f_Z}) 
admits the representation
\[
\widehat{f}_{\mathbf{z},\lambda}(x) = \sum_{i=1}^n \sqrt{\beta(x_i)} y_i \, \alpha_i(x),
\]
where  $\alpha(x) = (\alpha_1(x),\alpha_2(x),\ldots,\alpha_n(x)) \in \mathbb{R}^n$ is given by
\begin{center}
    $\alpha(x) = \frac{1}{n}\, g_\lambda\!\Big(\frac{1}{n}\mathcal B^{1/2} \mathbf{K}\mathcal B^{1/2}\Big)\,\mathcal B^{1/2} \,k_{x,n}.$
\end{center}
\end{proposition}

\noindent
\begin{proof}     
Recall that
\begin{center}
    $\widehat{f}_{\mathbf{z},\lambda}(x) = \widehat{C}_\lambda^\beta \phi(x).$
\end{center}

\noindent
From (\ref{c-equation}), we have

\begin{equation}\label{Flammid}
    \widehat{f}_{\mathbf{z},\lambda}(x) 
    = \widehat{C}_{YX}^\beta \, g_\lambda(\widehat{C}_X^\beta)\,\phi(x) = \Big( \frac{1}{n}\sum_{i=1}^n   \beta(x_i) \, y_i \otimes \phi(x_i) \Big)\,
       g_\lambda(\widehat{C}_X^\beta)\,\phi(x).
\end{equation}   

\noindent
Define the operator $Z : \mathcal{H} \to \mathbb{R}^n$ by
\begin{center}
    $Z f = \big( f(x_i)\sqrt{\beta(x_i)} \big)_{i=1}^n.$
\end{center}

\noindent
For any $f \in \mathcal{H}$ and $u \in \mathbb{R}^n$, the adjoint operator $Z^* : \mathbb{R}^n \to \mathcal{H}$ is given by
\[
    Z^*(u) = \sum_{i=1}^n \sqrt{\beta(x_i)}\, u_i \, \phi(x_i), 
    \qquad u = (u_1,\ldots,u_n)^T \in \mathbb{R}^n.
\]

\noindent
Therefore, we can define the operator $Z^*Z:\mathcal{H} \rightarrow\mathcal{H}$ by
\[Z^*Z(f) = \sum_{i=1}^n \beta(x_i) f(x_i)\,\phi(x_i).
\]

\noindent  
Also, $\widehat{C}_X^\beta$ can be written as
\[
\widehat{C}_X^\beta f
=
\frac{1}{n}\sum_{i=1}^n
\left( \beta(x_i)\,\phi(x_i) \otimes \phi(x_i) \right) f
=
\frac{1}{n}\sum_{i=1}^n
f(x_i)\,\beta(x_i)\,\phi(x_i).
\]

\noindent
Substituting $\widehat{C}_X^\beta = \frac{1}{n} Z^*Z$ in (\ref{Flammid}), we obtain
\[
\widehat{f}_{\mathbf z,\lambda}(x)=
\frac{1}{n}\sum_{i=1}^n
\big\langle
g_\lambda \Big(\frac{1}{n}Z^*Z\Big)\phi(x),
\phi(x_i)
\big\rangle_{\mathcal H}
y_i\,\beta(x_i).
\]

\noindent
Define 
\begin{center}
    $ 
    \qquad 
    Y = (y_1,\ldots,y_n), 
    \qquad 
    Y' =\mathcal B^{1/2}Y^T.$
\end{center}

\noindent
Recall that
\begin{center}
    $Z f = \big(f(x_i)\sqrt{\beta(x_i)}\big)_{i=1}^n 
    = \big(\langle f, \phi(x_i)\sqrt{\beta(x_i)} \rangle_{\mathcal{H}}\big)_{i=1}^n\,$.
\end{center}

\noindent
Thus, we can express the considered approximation  as
\begin{equation}\label{algo-4.3eq}
    \widehat{f}_{\mathbf{z},\lambda}(x) 
    = Z\,\!\Big(\frac{1}{n} g_\lambda\!\Big(\frac{1}{n}Z^*Z\Big)\phi(x)\Big)\, Y' = \Big(\frac{1}{n} g_\lambda\!\Big(\frac{1}{n}Z Z^*\Big) Z\phi(x)\Big)\, Y'. 
\end{equation}

\noindent
Note that the second equality follows from the analysis presented in \cite{meunier2024optimal}.

\noindent
Now, for any $u \in \mathbb{R}^n$ we have
\[
ZZ^* u
=
\Bigg(
\sum_{j=1}^n
\sqrt{\beta(x_i)}\, k(x_i,x_j)\, \sqrt{\beta(x_j)}\, u_j
\Bigg)_{i=1}^n
=
\mathcal B^{1/2}\, \mathbf K\, \mathcal B^{1/2}\, u .
\]

\noindent
Substituting  $ZZ^* =\mathcal B^{1/2} \mathbf{K}\mathcal B^{1/2}$ in  (\ref{algo-4.3eq}), we obtain
\begin{center}
    $\widehat{f}_{\mathbf{z},\lambda}(x) 
    = \Big( \frac{1}{n} g_\lambda\!\big( \frac{1}{n}\mathcal B^{1/2} \mathbf{K}\mathcal B^{1/2} \big)\, Z\phi(x) \Big)\, Y'.$
\end{center}

\noindent
Note that
\begin{center}
    $Z\phi(x) = \big( \langle \phi(x), \phi(x_i)\sqrt{\beta(x_i)} \rangle_{\mathcal H} \big)_{i=1}^n 
    = \big( \sqrt{\beta(x_i)}\, k(x,x_i) \big)_{i=1}^n =\mathcal B^{1/2} k_{x,n}.$
\end{center}

\vspace{0.4em}

\noindent
Hence, the approximation can be written in the matrix form as
\begin{equation}\label{predictor}
    \widehat{f}_{\mathbf{z},\lambda}(x) 
    = \Big( \frac{1}{n} g_\lambda\!\big( \frac{1}{n}\mathcal B^{1/2} \mathbf{K}\mathcal B^{1/2} \big)\,\mathcal B^{1/2} k_{x,n} \Big)\, Y',
\end{equation}

\noindent
that finally leads to the required representation
\[    \widehat{f}_{\mathbf{z},\lambda}(x) = \sum_{i=1}^n \sqrt{\beta(x_i)}  y_i \, \alpha_i(x),\]
where the coefficient vector $\alpha(x) \in \mathbb{R}^n$ is given by
\begin{center}
    $\alpha(x) = \frac{1}{n}\, g_\lambda\!\Big(\frac{1}{n}\mathcal B^{1/2} \mathbf{K}\mathcal B^{1/2}\Big)\,\mathcal B^{1/2} k_{x,n}.$
\end{center}
\end{proof}

\noindent
{Note that assuming $\beta$ to be in some RKHS, one can  follow \cite{IWRLS1} (see also \cite{ALPHA-SELECTION}) and approximate the vector $(\beta(x_1),\ldots,\beta(x_n))$ by means of the so-called Kernel unconstrained Least-Squares Importance Fitting (KuLSIF), in which the corresponding approximation  is given as
\begin{equation}\label{beta}
    (\beta_{\alpha}(x_1),\ldots,\beta_{\alpha}(x_n)) = (\mathbf{K} + \alpha nI)^{-1}\,\bar{F},
\end{equation}
where 
\[
    \bar{F} = \bigg( \frac{n}{m} \sum_{i=1}^m k(x'_i, x_j) \bigg)_{j=1}^n,
\]
$\{x'_i\}_{i=1}^m$ are unlabeled inputs sampled from the marginal target distribution $q_X(x)$, and $k$ is not necessary the same as in (\ref{multk}).}

\section{LEARNING RATES AND ERROR BOUNDS}\label{sec:proof}
\noindent
In this section, we provide a theoretical analysis of the proposed framework for general regularization under a general source condition. 
Specifically, we establish \emph{optimal convergence rates} for the estimator in both the vRKHS-norm and the $L_2(q_X, Y)$-norm. 
Before stating the key assumptions of our analysis, we define the covariance operator
$C_X : \mathcal{H} \to \mathcal{H}$ by
\[
C_X = C_X^\beta
:= \mathbb{E}_{p}\!\left[(\phi(x)\otimes\phi(x))\,\beta(x)\right]
= \mathbb{E}_{q}\!\left[\phi(x)\otimes\phi(x)\right],
\]
where the expectations $\mathbb{E}_p$ and $\mathbb{E}_q$ are taken with respect to the
marginals $p_X$ and $q_X$, respectively.

\begin{assumption}\label{kernel-bound}
Let $k : X \times X \to \mathbb{R}$ be the reproducing kernel of a scalar-valued RKHS $\mathcal{H}$, satisfying
\[
k(x,x) \le \kappa^2 \quad \mathrm{for\ all}\ x \in X .
\]
\end{assumption}

\begin{assumption}\label{bridingassumption}
It is assumed that there exists a reweighting function $\beta : X \to [0,\infty)$ such that 
\[
dq_X(x) = \beta(x)\, dp_X(x),
\]
and a constant $b > 0$ for which $\beta$ is bounded, i.e.,
\[
\beta(x) \leq b, \quad \forall\, x \in X.
\]

\end{assumption}

\begin{assumption}\label{bound-Y-asssumption}
There exists a constant $M' > 0$ such that,  $\|y\|_Y \le M'$ for all $y \in Y$.

\end{assumption}

{
To analyze error bounds and establish faster learning rates for learning algorithms, it is crucial to exploit the smoothness of the target function, which is typically characterized through a source condition.  
From \cite{meunier2024optimal}, the classical Hölder source condition is encompassed by assuming that the target regression operator satisfies
\[
C_\ast \in \mathcal S_2(\mathcal H^\alpha, Y),
\quad \mathrm{for \, some } \; \alpha \ge 0 .
\]
Furthermore, by Lemma~\ref{Lemma:3}, the operator $C_\ast\in \mathcal S_2(\mathcal H^\alpha, Y)$ admits the representation
\[
C_\ast = C_o (C_X)^\alpha,
\]
for some operator $C_o \in \mathcal S_2(\mathcal H, Y)$.
Motivated by this representation, we introduce a general source condition of the form
\[
C_\ast = C_o \varphi(C_X),
\]
where $\varphi$ is a general index function (see~\cite{Gupta_2025}).
}

\begin{assumption}{(General source condition)}\label{assump:source}
{
We assume that the operator $C_\ast$ belongs to the Hilbert--Schmidt class
$\mathcal S_2(\mathcal H^{\varphi},Y)$, where $\varphi$ is the index function and  the space $\mathcal H^{\varphi}$ is defined as
\[
\mathcal H^{\varphi}
:= \left\{ f \in \mathcal H \;:\; f = \varphi(C_X) g \quad \mathrm{for \, some } \; g \in \mathcal H \right\}.
\]
}
\end{assumption}

\noindent
{
Note that the space $\mathcal H^{\varphi}$ is a Hilbert space when endowed with the inner product
\[
\big\langle \varphi(C_X) g_1,\, \varphi(C_X) g_2 \big\rangle_{\mathcal H^{\varphi}}
:= \langle g_1, g_2 \rangle_{\mathcal H},
\qquad g_1, g_2 \in \mathcal H.
\]
}
{
Consequently, the family $
\{\varphi(\mu_i) f_i\}_{i \ge 1}$
forms an orthonormal basis of $\mathcal H^{\varphi}$.
}

\begin{assumption}\label{eigendecay}
Let $\{(\mu_i,f_i)\}_{i\in\mathbb{N}}$ be the eigenpairs of $C_X$ with strictly positive eigenvalues $\mu_i>0$, and corresponding eigenfunctions $\{f_i\}_{i\in\mathbb{N}}$ forming an orthonormal basis of $\mathcal H$. Assume that for some $\mathbf{b} > 1$, the eigenvalues satisfy
\[
i^{-\mathbf{b}} \;\lesssim\; \mu_i \;\lesssim\; i^{-\mathbf{b}},
\qquad \forall\, i \in \mathbb{N}.
\]
\end{assumption}

\noindent
This assumption characterizes the spectral decay of the integral operator and allows us to control the \emph{effective dimension}
\[
\mathcal{N}(\lambda) 
:= 
{Tr}\!\left(C_X\big(C_X + \lambda I\big)^{-1}
\right),
\]
which, under the above condition, admits the bound
\begin{center}
$\mathcal{N}(\lambda) \;\lesssim\; \lambda^{-\frac{1}{\mathbf{b}}}.
$
\end{center}

\noindent
We next state several key lemmas that are required for our analysis. The proofs of some of these lemmas are provided in Section~\ref{Auxillary_results}.

\begin{lemma}\label{Lemma:3}
Let \(C \in \mathcal{S}_2(\mathcal{H},Y)\) and $\varphi$ be an index function. Then
\[
C \in \mathcal{S}_2(\mathcal{H}^{\varphi},Y)
\quad\iff\quad
\text{there exists an operator } \, C_o \in \mathcal{S}_2(\mathcal{H},Y)
\ \text{such that}\
C = C_o\,\varphi(C_X).
\] 
\end{lemma}

\begin{lemma}\label{Lemma:1}

Suppose that assumptions~\ref{kernel-bound} and~\ref{bridingassumption} hold. Then, with probability at least \(1-\delta\), we have    \[
    \big\|(C_{X}+\lambda I)^{-1/2} (C_X -\widehat{C}_X^\beta)\big\|_{op} \leq 2\left(\frac{2b\kappa^2}{\sqrt \lambda\, n} +\sqrt{\frac{b\kappa^2\mathcal{N}(\lambda)}{n}} \right) \log \left(\frac{2}{\delta}\right).
\]
\end{lemma}

\begin{lemma}\label{Lemma:2}
Suppose that assumptions~\ref{kernel-bound} and~\ref{bridingassumption} hold. Then, with probability at least \(1-\delta\), the following operator norm bound holds:

\[
\left\|(C_X+\lambda I)(\widehat{C}_X^{\beta}+\lambda I)^{-1}\right\|_{{op}}
\;\leq\;
2\left[
\left(
\frac{\mathcal{B}_{n,\lambda}\,\log\!\left(\frac{2}{\delta}\right)}{\sqrt{\lambda}}
\right)^2
+1
\right],
\]
where
\[
\mathcal{B}_{n,\lambda}
\;=\;
2\left(
\frac{2b\kappa^2}{\sqrt{\lambda}\,n}
+
\sqrt{\frac{b\kappa^2\mathcal{N}(\lambda)}{n}}
\right).
\]
\end{lemma}

\begin{lemma}\label{lemma:5}
Let $0<l<1$ and $0<\delta<1$. Assume that assumptions~\ref{kernel-bound} and~\ref{bridingassumption} hold, and that \(\lambda > 0\) satisfies
\begin{equation}\label{eq:condition_on_lambda}
\mathcal{B}_{n,\lambda}
\log\!\left(\frac{2}{\delta}\right)\leq \frac{\sqrt{\lambda}}{2}.
\end{equation}
Then, with probability at least $1-\delta$, the following bound holds:
\[
\left\|
(C_X+\lambda I)^l(\widehat{C}^\beta_X+\lambda I)^{-l}
\right\|_{{op}}
\;\le\;
2^{\,l}.
\]
\end{lemma}

\begin{lemma}\label{L_2(lemma-A)}
Suppose that {assumptions~\ref{kernel-bound}--\ref{assump:source}} are satisfied.
If  the qualification $\nu$ of the regularization family covers the index function \(\varphi(t)\sqrt{t}\) and \eqref{eq:condition_on_lambda} holds,
then with probability at least \(1-\delta\), we have

\[\big\|(C_*-\bar{C}_{\lambda}^\beta)\, C_X^{1/2} \big\|_{\mathcal S_2(\mathcal H,Y)} \leq\;
\left(
\frac{A_1}{n}\sqrt{\frac{1}{\lambda}}
+ A_2\sqrt{\frac{\mathcal{N}(\lambda)}{n}}
\right) \log\left(\frac{2}{\delta}\right)
+ A_3\,\sqrt{\lambda}\varphi(\lambda).\]

Here $\bar{C}_{\lambda}^{\beta}
= C_* \widehat{C}_X^{\beta} g_{\lambda}(\widehat{C}_X^{\beta})$,
and the constants $A_1, A_2, A_3$ are given by
\begin{align*}
A_1 &:= 2^{\,2\nu-\bar{\nu}+2}(\gamma_\nu+\gamma)\,
      \max\!\left(1,\frac{1}{c}\right)\,
      \bar{\nu}\, b\, \kappa^2\,
      \| C_{*}\|_{\mathcal S_2(\mathcal H,Y)},\\[4pt]
A_2 &:= 2^{\,2\nu-\bar{\nu}+1}(\gamma_\nu+\gamma)\,
      \max\!\left(1,\frac{1}{c}\right)\,
      \bar{\nu}\,\sqrt{b}\,\kappa\,
      \|C_{*}\|_{\mathcal S_2(\mathcal H,Y)},\\[4pt]
A_3 &:= 2^{\,2\nu-\bar{\nu}}(\gamma_\nu+\gamma)\,
      \max\!\left(1,\frac{1}{c}\right)\,
      \|C_{o}\|_{\mathcal S_2(\mathcal H,Y)},\\[4pt]
\end{align*}
where $\bar{\nu} := \lfloor \nu -1/2 \rfloor$ denotes the greatest integer
less than or equal to~$\nu -1/2$.

\end{lemma}

\begin{lemma}\label{L2:bound-B}
Suppose that assumptions~\ref{kernel-bound}--\ref{assump:source} are satisfied and that equation~\eqref{eq:condition_on_lambda} holds. Then, with probability at least \(1-\delta\), we have
\begin{align*}
\bigl\| ( \widehat{C}_\lambda^\beta-\bar{C}_{\lambda}^\beta)\, C_X^{1/2} \bigr\|_{\mathcal S_2(\mathcal H,Y)}
\;&\leq\;
(B+D)\left(
\frac{10bM\kappa}{n\sqrt{\lambda}}
+ 5
M\sqrt{
\frac{b \; \mathcal{N}(\lambda)}{n}
}
\right)\log\!\left( \frac{2}{\delta} \right).
\end{align*}
\end{lemma}

\begin{lemma}[\cite{Pinelis1985}]\label{PINELIS:INEQUALITY} 
Let $(\Omega,\mathcal{F},\mathbb{P})$ be a probability space and let $\xi$ be a random variable on $\Omega$ taking values in a real separable Hilbert space $\mathcal{H}$. 
Assume that there exist two positive constants $L$ and $\sigma$ such that

\begin{equation}\label{moment-equation}
\mathbb{E}\!\left[\big\|\xi-\mathbb{E}[\xi]\big\|_{\mathcal{H}}^{m}\right]
\le \frac{1}{2} m!\,\sigma^{2} L^{m-2},
\qquad \forall\, m \ge 2 .
\end{equation}
Let $\xi_1,\ldots,\xi_n$ be i.i.d.\ copies of $\xi$. Then, for all $n \in \mathbb{N}$ and $0<\delta<1$, it holds that

\[
\mathbb{P}
\left(
\left\|
\frac{1}{n}\sum_{i=1}^{n} \xi_i - \mathbb{E}[\xi]
\right\|_{\mathcal{H}}
\le
2\left(\frac{L}{n} + \frac{\sigma}{\sqrt{n}}\right)
\log \left(\frac{2}{\delta}\right)
\right)
\ge 1-\delta .
\]
Moreover, the moment condition~(\ref{moment-equation}) holds provided that

\[
\|\xi(\omega)\|_{\mathcal{H}} \le \frac{L}{2} \quad {almost \,surely},
\quad {and} \quad
\mathbb{E}\!\left[\|\xi\|_{\mathcal{H}}^{2}\right] \le \sigma^{2}.
\]
\end{lemma}

\begin{lemma}[Lemma A.3, \cite{GUPTA_ACHA2025101745}]\label{Naveen-acha:lemma}
Let $A$ and $B$ be positive operators on a Hilbert space $\mathcal H$.
Then, for any integer $n \ge 1$ and any $\lambda > 0$, the following identity holds:
\begin{align*}
(B + \lambda I)^{-n} - (A + \lambda I)^{-n}
&= (B + \lambda I)^{-(n-1)}
\bigl[ (B + \lambda I)^{-1}
      - (A + \lambda I)^{-1} \bigr] \\
&\quad + \sum_{i=1}^{n-1}
(B + \lambda I)^{-i}
(A - B)
(A + \lambda I)^{-(n+1-i)} .
\end{align*}
\end{lemma}

\noindent
The following two theorems present the main error bounds of this section, providing estimates in the  \(L_{2}(q_X,{Y})\) norm (Theorem~\ref{Main-Theorem L2}) and in the vRKHS norm (Theorem~\ref{vrkhs:theorem:proof}), respectively.

\begin{theorem}\label{Main-Theorem L2}
 Suppose that {assumptions~\ref{kernel-bound}--\ref{assump:source}} are satisfied.
If \eqref{eq:condition_on_lambda} holds and the qualification $\nu$ of the regularization family covers the index function \(\varphi(t)\sqrt{t}\),
then with probability at least \(1-\delta\), we have
\[
\big\|f^*-\widehat{f}_{\textbf{z},\lambda} \big\|_{{L}_2(q_X,Y)}\leq \left(
\frac{A_1}{n}\sqrt{\frac{1}{\lambda}}
+ A_2\sqrt{\frac{\mathcal{N}(\lambda)}{n}}
\right) \log\left(\frac{2}{\delta}\right)
+ A_3\,\sqrt{\lambda}\,\varphi(\lambda).
\]

Here the constants $A_1,A_2,A_3$ are given by 
\begin{align*}
A_1
&= 2^{2\nu-\bar \nu+2}(\gamma_\nu+\gamma)\max\!\left(1,\frac{1}{c}\right)
\,\bar \nu\, b\,\kappa^{2}\,\|C_{*}\|_{\mathcal S_2(\mathcal H,Y)}
\;+\; 10\,b\,M\,\kappa\,(B+D), \\[6pt]
A_2
&= 2^{2\nu-\bar \nu+1}(\gamma_\nu+\gamma)\max\!\left(1,\frac{1}{c}\right)
\,\bar \nu\, \sqrt{b}\kappa\,\|C_{*}\|_{\mathcal S_2(\mathcal H,Y)}
\;+\; 5\,M\,\sqrt{b}\,(B+D), \\[6pt]
A_3
&= 2^{2\nu-\bar \nu}(\gamma_\nu+\gamma)\max\!\left(1,\frac{1}{c}\right)
\,\|C_{o}\|_{\mathcal S_2(\mathcal H,Y)},
\end{align*}
where $\bar{\nu} := \lfloor \nu -1/2 \rfloor$ denotes the greatest integer
less than or equal to~$\nu -1/2$.
\end{theorem}

\begin{proof}
    \noindent
We start by observing that
\begin{align*}
\big\|f^*-\widehat{f}_{\textbf{z},\lambda}\big\|_{{L}_2(q_X,Y)}
&=
\big\|C_*-\widehat{C}_{\lambda}^\beta \big\|_{\mathcal S_2( L_2(q_X),Y)} \\
&=
\big\|(C_*-\widehat{C}_{\lambda}^\beta)\, C_X^{1/2} \big\|_{\mathcal S_2(\mathcal H,Y)} .
\end{align*}

Applying the triangle inequality yields
\begin{align*}
\big\|(C_*-\widehat{C}_{\lambda}^\beta)\, C_X^{1/2} \big\|_{\mathcal S_2(\mathcal H,Y)}
&\leq
\underbrace{\big\|(C_*-\bar{C}_{\lambda}^\beta)\, C_X^{1/2} \big\|_{\mathcal S_2(\mathcal H,Y)}}_{\textbf{A}}  +
\underbrace{\big\|(\bar{C}_{\lambda}^\beta -\widehat{C}_\lambda^\beta)\, C_X^{1/2} \big\|_{\mathcal S_2(\mathcal H,Y)}}_{\textbf{B}} .
\end{align*}

\noindent
The bound for $Term ~ \textbf{A}$ can be obtained using Lemma \ref{L_2(lemma-A)} and bound for $Term~\textbf{B}$ can be obtained using Lemma \ref{L2:bound-B}, hence with probability at least $1-\delta$, we have

\[
 \big\|f^*-\widehat{f}_{\textbf{z},\lambda}\big\|_{{L}_2(q_X,Y)}\leq \left(
\frac{A_1}{n}\sqrt{\frac{1}{\lambda}}
+ A_2\sqrt{\frac{\mathcal{N}(\lambda)}{n}}
\right) \log\left(\frac{2}{\delta}\right)
+ A_3\,\sqrt{\lambda}\,\varphi(\lambda).
\]

\end{proof}

\begin{theorem}\label{vrkhs:theorem:proof}
Suppose that {assumptions~\ref{kernel-bound}--\ref{assump:source}} are satisfied.
If \eqref{eq:condition_on_lambda} holds and the qualification $\nu$ of the regularization family covers the index function \(\varphi(t)\),
then with probability at least \(1-\delta\), we have

\[
\big\|f^*-\widehat{f}_{\textbf{z},\lambda}\big\|_{\mathcal G}\leq \left(
\frac{A_1}{n\lambda}
+ A_2\sqrt{\frac{\mathcal{N}(\lambda)}{n\lambda}}
\right) \log\left(\frac{2}{\delta}\right)
+ A_3\,\varphi(\lambda).
\]

\noindent
Here the constants $A_1,A_2,A_3$ are given by 
\begin{align*}
A_1
&= 2^{2\nu- \bar \nu+2}(\gamma_\nu+\gamma)\max\!\left(1,\frac{1}{c}\right)
\,\bar \nu\, b\,\kappa^{2}\,\|C_{*}\|_{\mathcal S_2(\mathcal H,Y)}
\;+\; 2\sqrt{10}\,b\,M\,\kappa\,(B+D), \\[6pt]
A_2
&= 2^{2\nu- \bar \nu+1}(\gamma_\nu+\gamma)\max\!\left(1,\frac{1}{c}\right)
\, \bar \nu\, \sqrt{b}\kappa\,\|C_{*}\|_{\mathcal S_2(\mathcal H,Y)}
\;+\; \sqrt{10}\,\sqrt{b}\,M(B+D), \\[6pt]
A_3
&= 2^{2\nu- \bar \nu}(\gamma_\nu+\gamma)\max\!\left(1,\frac{1}{c}\right)
\,\|C_{o}\|_{\mathcal S_2(\mathcal H,Y)}.
\end{align*}
where $\bar{\nu} := \lfloor \nu  \rfloor$ denotes the greatest integer
less than or equal to~$\nu $.
\end{theorem}

\begin{proof}
    \noindent
We have
\begin{align*}
\big\|f^*-\widehat{f}_{\textbf{z},\lambda}\big\|_{\mathcal G}
&=
\big\|C_*-\widehat{C}_{\lambda}^\beta\, \big\|_{\mathcal S_2(\mathcal H,Y)} \leq
\underbrace{\big\|C_*-\bar{C}_{\lambda}^\beta\, \big\|_{\mathcal S_2(\mathcal H,Y)}}_{\textbf{A}}  +
\underbrace{\big\|\bar{C}_{\lambda}^\beta -\widehat{C}_\lambda^\beta\, \big\|_{\mathcal S_2(\mathcal H,Y)}}_{\textbf{B}}. 
\end{align*}

\noindent
The bound for Term~\textbf{A} can be obtained by following the steps of
Lemma~\ref{L_2(lemma-A)}. Let $\bar \nu$ denote the greatest integer less than $\nu$.
Then, with probability at least $1-\delta$, we obtain

\[
\begin{aligned}
\big\|C_*-\bar{C}_{\lambda}^\beta\big\|_{\mathcal  S_2(\mathcal H,Y)}
\leq\;&\max\!\left(1,\frac{1}{c}\right)
\Bigg(
\frac{2^{\,2\nu-\bar \nu}(\gamma_\nu+\gamma)\,
      4\,\bar\nu\,b\,\kappa^2\,
      \|C_{*}\|_{\mathcal S_2(\mathcal H,Y)}}{n\lambda}
\\
&\quad
+ 2^{\,2\nu-\bar\nu}(\gamma_\nu+\gamma)\,
      2\,\bar\nu\,\sqrt{b}\kappa\,
      \|C_{*}\|_{\mathcal S_2(\mathcal H,Y)}
      \sqrt{\frac{\mathcal{N}(\lambda)}{n\lambda}}
\Bigg)
\log\!\left(\frac{2}{\delta}\right)
\\
&\quad
+ 2^{\,2\nu-\bar\nu}(\gamma_\nu+\gamma)\,
      \max\!\left(1,\frac{1}{c}\right)\,
      \|C_{o}\|_{\mathcal S_2(\mathcal H,Y)}\,
      \varphi(\lambda).
\end{aligned}
\]

To obtain bound for $Term~$\textbf{B}, we observe
\begin{align*}
\bigl\|  \widehat{C}_\lambda^\beta-\bar{C}_{\lambda}^\beta\,  \bigr\|_{\mathcal  S_2(\mathcal H,Y)}
&=
\bigl\|
( \widehat{C}_{Y,X}^\beta-{C_*}\widehat C_X^\beta)\,
g_\lambda(\widehat{C}_X^\beta)\,
\bigr\|_{\mathcal  S_2(\mathcal H,Y)} .
\end{align*}

\noindent
Next, we split the operator as
\begin{align*}
&\bigl\|
( \widehat{C}_{Y,X}^\beta-{C_*}\widehat C_X^\beta)\,
g_\lambda(\widehat{C}_X^\beta)\,
\bigr\|_{\mathcal S_2(\mathcal H,Y)} \\
&\leq
\underbrace{\bigl\|
( \widehat{C}_{Y,X}^\beta-{C_*}\widehat C_X^\beta)\,
(C_{X}+\lambda I)^{-1/2}
\bigr\|_{\mathcal S_2(\mathcal H,Y)}}_{A}
\\
&\hspace{9.5em}\times
\underbrace{\bigl\|
(C_{X}+\lambda I)^{1/2}
(\widehat{C}^{\beta}_{X}+\lambda I)^{-1/2}
\bigr\|_{op}}_{B}
\;
\underbrace{\bigl\|
(\widehat{C}^{\beta}_{X}+\lambda I)^{1/2}
g_\lambda(\widehat{C}_X^\beta)
\bigr\|_{op}}_{C}.
\end{align*}

\noindent
The bound for $Term~A,B$ are already given in proof of Lemma \ref{L2:bound-B}. For $Term~C$, we can write

\[
C = {sup}_\sigma \; | (\sigma +\lambda)^{1/2} g_\lambda(\sigma)| \leq {sup}_\sigma \; \left| \frac{ (\sigma +\lambda) g_\lambda(\sigma)}{(\sigma +\lambda)^{1/2}}\right| \leq \frac{B+D}{\sqrt{\lambda}}.   \]

\noindent
Hence, with probability at least $1-\delta$, we obtain

\begin{align*}
\bigl\|  \widehat{C}_\lambda^\beta-\bar{C}_{\lambda}^\beta\,  \bigr\|_{\mathcal  S_2(\mathcal H,Y)}
\;&\leq\;
\sqrt{\frac{5}{2}}(B+D)\left(
\frac{4bM\kappa}{n{\lambda}}
+ 2
M\sqrt{
\frac{b \; \mathcal{N}(\lambda)}{n\lambda}
}
\right)\log\!\left( \frac{2}{\delta} \right)
.
\end{align*}

\noindent
Finally, combining the bounds for Term~\textbf{A} and Term~\textbf{B} yields the desired result.

\end{proof}

The following two results establish the convergence rates in both the \(L_{2}(q_X,{Y})\) norm and the vRKHS norm, covering the general source condition (Theorem~\ref{Rates-theorem-general}) as well as Hölder's source condition (Corollary~\ref{Holder-corollary}).

\begin{theorem}\label{Rates-theorem-general}
Suppose that {assumptions~\ref{kernel-bound}--\ref{eigendecay}} hold and  $ \delta \in (0,  1)$. 
Assume $\lambda \in (0,1]$ satisfies $\lambda = \Psi^{-1}(n^{-\frac{1}{2}})$, 
where $\Psi(t) = t^{\frac{1}{2} + \frac{1}{2\mathbf{b}}} \varphi(t)$.

\begin{enumerate}
\renewcommand{\labelenumi}{(\roman{enumi})}

\item
If the index function $\varphi$ satisfies the qualification of the regularization family, then with probability at least $1-\delta$,
\[
\left\|f^*-\widehat{f}_{\mathbf z,\lambda}\right\|_{\mathcal G}
\leq
c_1\,
\varphi\!\left(\Psi^{-1}\!\left(n^{-\frac{1}{2}}\right)\right)
\log\!\left(\frac{2}{\delta}\right),
\]
for some positive constant $c_1$.

\item
If the index function $\varphi(t)\sqrt{t}$ satisfies the qualification of the regularization family, then with probability at least $1-\delta$,
\[
\left\|f^*-\widehat{f}_{\mathbf z,\lambda}\right\|_{L_2(q_X,Y)}
\leq
c_2\,
\left(\Psi^{-1}\!\left(n^{-\frac{1}{2}}\right)\right)^{\frac{1}{2}}
\varphi\!\left(\Psi^{-1}\!\left(n^{-\frac{1}{2}}\right)\right)
\log\!\left(\frac{2}{\delta}\right),
\]
for some positive constant $c_2$.

\end{enumerate}

\end{theorem}

\begin{corollary}\label{Holder-corollary}
Suppose that Assumptions~\ref{kernel-bound}--\ref{eigendecay} hold and let
$\delta \in (0,1)$. 
Assume that the Hölder's source condition holds, i.e.,
$\varphi(t)=t^r$ for some $r>0$,
and choose the regularization parameter
$\lambda = n^{-\frac{\mathbf b}{2\mathbf b r+\mathbf b+1}}$.
Then, with probability at least $1-\delta$, the following statements hold.
\begin{enumerate}
\renewcommand{\labelenumi}{(\roman{enumi})}

\item
If $\varphi$ is covered by the qualification of the regularization family, then for $0<r<\nu$,
\[
\left\|f^*-\widehat{f}_{\mathbf z,\lambda}\right\|_{\mathcal G}
\le
c_1\,
n^{-\frac{b r}{2 b r+b+1}}
\log\!\left(\frac{2}{\delta}\right),
\]
for some positive constant $c_1$.

\item
If $\varphi(t)\sqrt{t}$ is covered by the qualification of the regularization family, then for $0<r<\nu-\frac12$,
\[
\left\|f^*-\widehat{f}_{\mathbf z,\lambda}\right\|_{L_2(q_X,Y)}
\le
c_2\,
n^{-\frac{2 b r+b}{4 b r+2 b+2}}
\log\!\left(\frac{2}{\delta}\right),
\]
for some positive constant $c_2$.

\end{enumerate}

\end{corollary}

\begin{remark}
Under the polynomial eigenvalue decay assumption, the convergence rates obtained in
Theorem~\ref{Rates-theorem-general} and Corollary~\ref{Holder-corollary} coincide with the
minimax optimal rates reported in \cite{rastogi2017optimal} and
\cite{blanchard2016optimalratesHolders} for the general and Hölder's source condition
respectively. Moreover, without Assumption~\ref{eigendecay}, i.e., in the absence of a
polynomial eigenvalue decay condition, similar estimates as in Theorem~10 of \cite{regularization} can still be derived for
general regularization schemes. 
This confirms that our estimator attains the optimal rate of convergence under the stated assumptions.
\end{remark}

\section{Aggregation in Regularized Vector-Valued Learning  with covariate shift
}\label{Aggregation}

Careful tuning of scalar multiplicative kernels $k(x,t)$ and  regularization parameters $\lambda$ in (\ref{predictor}) plays a crucial role, but may be computationally expensive due to the lack of prior knowledge. In the scalar setting, an aggregation approach has been developed to handle parameter selection under covariate shift \cite{DBLP:conf/iclr/DinuHBNHEMPHZ23}. Motivated by that idea, we extend the approach \cite{DBLP:conf/iclr/DinuHBNHEMPHZ23} to the infinite-dimensional setting.

Let $\{f_{\mathbf{z},j}\}_{j=1}^{l}$ be the models constructed by some fairly generic procedures from the source data sample 
$\mathbf{z} =$ $\{(x_i, y_i)\}_{i=1}^n \in X \times Y$ and (possibly approximate) values $\{\beta(x_i)\}_{i=1}^n $ of the reweighting function  $\beta(x) = \frac{dq_X(x)}{dp_X(x)}$. Considering their linear combinations

\begin{center}
$    
 f(\mathbf{c};x) = \sum_{j=1}^l c_j f_{\mathbf{z},j}(x), \quad  c_j \in \mathbb{R}, \quad \mathbf{c} = (c_1,...,c_l),
$
\end{center}

we are interested in approximating the minimizer $\mathbf{c} =\mathbf{c}^*$
\begin{equation}\label{eq4.5.3}
  \mathbf{c}^*  = \text{argmin}_{\mathbf{c} \in \mathbb{R}^l } \big\|f(\mathbf{c};\cdot) - f^*(\cdot)\big\|^2_{L_2(q_X,Y)} .
\end{equation}
To an extent, the problem of the choice of the best model among the set $\{f_{\mathbf{z},j}\}_{j=1}^{l}$ can be reduced to the minimization in  (\ref{eq4.5.3}) because it is clear that 
\begin{equation}\label{eq4.5.4}
  \min_{\mathbf{c} \in \mathbb{R}^l } \big\|f(\mathbf{c};\cdot) - f^*(\cdot)\big\|^2_{L_2(q_X,Y)} \leq \min_{j\in \{1,...,l \}} \big\|f_{\mathbf{z},j}(\cdot) - f^*(\cdot)\big\|^2_{L_2(q_X,Y)} .
\end{equation}
It is easy to see that the minimizer $\mathbf{c} =\mathbf{c}^*$ solves the system $\mathbf{G}_l\mathbf{c} = \mathbf{g}_l$ with the Gram matrix $ \mathbf{G}_l = \left( \langle  f_{\mathbf{z},j}(\cdot),f_{\mathbf{z},k}(\cdot)\rangle_{L_2(q_X,Y)} \right)_{j,k=1}^l $ and the vector $\mathbf{g}_l = \left( \langle f^*(\cdot) ,f_{\mathbf{z},j}(\cdot)\rangle_{L_2(q_X,Y)} \right)_{j=1}^l $.

Observe, however, that neither the matrix $\mathbf{G}_l$ nor the vector $\mathbf{g}_l$ is accessible, because $q_X$ and $f^*(\cdot)$ are unknown. But they can be well approximated using $\{\beta(x_i)\}_{i=1}^n $ and the source data sample 
$\mathbf{z} =$ $\{(x_i, y_i)\}_{i=1}^n \in X \times Y$  drawn from an unknown joint probability distribution $p(x,y)$.

To discuss this, we recall (see e.g., \cite{meunier2024optimal}) that the regression function $f^*$ allows a representation via the so-called Markov kernel $\mathcal{K}(dy,x)$ as follows
{$$f^*(x) = \int_{Y}y \mathcal{K}(dy,x).$$} Then,
\begin{align} \label{eq4.5.5}
\langle f^*(\cdot),f_{\textbf{z},j}(\cdot)\rangle_{L_2(q_X,Y)} &= \biggl \langle \int_{Y}y \mathcal{K}(dy,\cdot),f_{\textbf{z},j}(\cdot)\biggr \rangle_{L_2(q_X,Y)} \notag \\ 
&= \int_{X}\biggl\langle \int_{Y}y \mathcal{K}(dy,x),f_{\textbf{z},j}(x)\biggr\rangle_{Y}dq_X(x) \notag \\ 
&=  \int_{X}\int_{Y}\langle y,f_{\textbf{z},j}(x)\rangle_{Y}\beta(x) \mathcal{K}(dy,x)dp_X(x) \notag \\  
&= \mathbb{E}_p\bigl[\beta(x) \langle y,f_{\textbf{z},j}(x)\rangle_{Y}\bigr].
\end{align}

The above expectation can be approximated by Monte-Carlo integration using the samples from $\mathbf{z} =\{(x_i, y_i)\}_{i=1}^n \in X \times Y$. By standard concentration argument, such as Lemma \ref{PINELIS:INEQUALITY}, the error of that integration formula, with probability at least $1 -\delta$, is of order $(n)^{-1/2}\log(\frac{1}{\delta})$, that is, 
\begin{align} \label{eq:monte1}
\langle f^*(\cdot),f_{\textbf{z},j}(\cdot)\rangle_{L_2(q_X,Y)} = \frac{1}{n}\sum_{i=1}^{n} \beta(x_i) \langle y_i,f_{\textbf{z},j}(x_i)\rangle_{Y} + \mathcal{O}\left( \frac{\log (\frac{1}{\delta})}{\sqrt{n} }\right).
\end{align}

For a similar reason, with probability at least $1 -\delta$, one has
\begin{align} \label{eq:monte3}
\langle  f_{\textbf{z},j}(\cdot),f_{\textbf{z},k}(\cdot)\rangle_{L_2(q_X,Y)}  = \frac{1}{n}\sum_{i=1}^{n} \beta(x_i) \langle f_{\textbf{z},j}(x_i),f_{\textbf{z},k}(x_i)\rangle_{Y}+ \mathcal{O}\left( \frac{\log(\frac{1}{\delta})}{\sqrt{n} }\right).
\end{align}
The above formulae allow for efficient and accurate calculations of the entries of the matrix $\mathbf{G}_l$ and the right-hand side vector $\mathbf{g}_l$ of the linear system defining the minimizer $\mathbf{c} = \mathbf{c}^*$ of $\big\|f(\mathbf{c};\cdot) - f^*(\cdot)\big\|^2_{L_2(q_X,Y)}$. Namely, the corresponding estimators for $\mathbf{G}_l$ and $\mathbf{g}_l$ can be given respectively as follows:
\begin{equation}\label{agg G}
\widetilde{\mathbf{G}}_l 
= \left( \frac{1}{n} \sum_{i=1}^{n}  
   \, \beta(x_i) \langle f_{\mathbf{z},j}(x_i),f_{\mathbf{z},k}(x_i)\rangle_{Y} \right)_{j,k=1}^{l},
\end{equation}

\begin{equation}\label{agg g}
\tilde{\mathbf{g}}_l 
= \left( \frac{1}{n} \sum_{i=1}^{n} 
   \, \beta(x_i) \langle y_i,f_{\mathbf{z},j}(x_i)\rangle_{Y} \right)_{j=1}^{l}.
\end{equation}

\vspace{0.5em}

\begin{proposition}
Let $\mathbf{G}_l,\mathbf{g}_l,\tilde{\mathbf{G}}_l$ and $\tilde{\mathbf{g}}_l$ be defined as above. Then with probability at least $1-\delta$, the following bounds hold:
\begin{align}
\big\| \textbf{g}_l - \tilde{\textbf{g}}_l \big\|_{\mathbb{R}^l} 
&\leq c \, \sqrt{l}\log \left( \frac{1}{\delta} \right) n^{-1/2} , \label{eq:g-bound}\\[0.6em]
\big\| \textbf{G}_l - \widetilde{\textbf{G}}_l \big\|_{op} 
&\leq c \, l\,\log  \left( \frac{1}{\delta} \right) n^{-1/2}  
, \label{eq:G-bound}
\end{align}
where $c$ is some generic constant.
\end{proposition}

\noindent

\begin{proof}
If $G_{j,k}, \tilde{G}_{j,k}$ denote the $(j,k)$-th entries of the matrices $\textbf{G}_l$ $\tilde{\textbf{G}}_l$, then from \eqref{eq:monte3} with probability at least $1-\delta$, it follows that
\begin{equation*}
\max_{1 \leq j,k \leq n} 
   \big| \tilde{G}_{j,k} - G_{j,k} \big|
   \;\leq\; c \; \log  \left( \frac{1}{\delta} \right) n^{-1/2}  .
\end{equation*}

\noindent
Using this bound, we obtain
\begin{equation*}
\big\| \widetilde{\textbf{G}}_l - \textbf{G}_l \big\|_{op} 
   \;\leq c\; l  \, \log  \left( \frac{1}{\delta} \right) n^{-1/2}  .
\end{equation*}

\noindent
The same reason, together with  \eqref{eq:monte1}, gives
\begin{center}
$\big\| \tilde{\textbf{g}}_l - \textbf{g}_l \big\|_{\mathbb{R}^l} 
   \;\leq c\; \log  \left( \frac{1}{\delta} \right) n^{-1/2} .$
\end{center}

\end{proof}

Note that with matrix $\widetilde{\mathbf{G}}_l$ at hand one can easily check whether or not it is well-conditioned and $\widetilde{\mathbf{G}}_l^{-1}$ exists. If it is not the case, then some approximants $f_{\mathbf{z},j}$ can be withdrawn from consideration without loss. Therefore, we assume that $\| \widetilde{\mathbf{G}}_l^{-1} \|_{op} = \mathcal{O}(1)$, as well as that $\max_j\|f_{\mathbf{z}, j}\|_{L_2(q_X,Y)} = \mathcal{O}(1)$.

We assume also that $n$ is large enough
s.t. with probability at least $1-\delta$ we have

\begin{equation}\label{largens.t.}
\big\|\mathbf{G}_l-\widetilde{\mathbf{G}}_l\big\|_{op} < \frac{1}{\|\widetilde{\mathbf{G}}_l^{-1}\|_{op}} \quad.
\end{equation}

\noindent
Next we can express $\mathbf{G}_l^{-1}$ as

\[
\mathbf{G}_l^{-1} = \widetilde{\mathbf{G}}_l^{-1}(\mathbf{G}_l\widetilde{\mathbf{G}}_l^{-1})^{-1} 
= \widetilde{\mathbf{G}}_l^{-1}\big(I - (I - \mathbf{G}_l\widetilde{\mathbf{G}}_l^{-1})\big)^{-1} 
= \widetilde{\mathbf{G}}_l^{-1}\big(I - (\widetilde{\mathbf{G}}_l-\mathbf{G}_l)\widetilde{\mathbf{G}}_l^{-1}\big)^{-1}.
\]

\noindent
Then from (\ref{largens.t.}) it follows that the Neumann series for \((I - (\widetilde{\mathbf{G}}_l-\mathbf{G}_l)\widetilde{\mathbf{G}}_l^{-1})^{-1}\) converges, and we obtain the following bound:
\[
\|\mathbf{G}_l^{-1}\|_{op} 
\leq \frac{\|\widetilde{\mathbf{G}}_l^{-1}\|_{op}}{1 - \|\widetilde{\mathbf{G}}_l^{-1}\|_{op} \|\mathbf{G}_l-\widetilde{\mathbf{G}}_l\|_{op}}
= \mathcal{O}(1).
\]

\begin{theorem}\label{Tagg}
Consider $f(\widetilde{\mathbf{c}};\cdot)  = \sum_{j=1}^l \tilde{c}_j f_{\mathbf{z}, j}(\cdot)$, 
where $\widetilde{\mathbf{c}}= (\tilde{c}_1, \tilde{c}_2, \ldots, \tilde{c}_l) = \widetilde{\mathbf{G}}_l^{-1}\tilde{\mathbf{g}}_l$, and assume that (\ref{largens.t.}) holds, then with probability at least $1-\delta$ we have
\[
\big\|f^*(\cdot) - f(\widetilde{\mathbf{c}};\cdot)\big\|_{L_2(q_X,Y)} = \min_{c_j} \Big\| f^*(\cdot) - \sum_{j=1}^l c_j f_{\mathbf{z},j}(\cdot)\Big\|_{L_2(q_X,Y)}
+ \mathcal{O}\left(l^{\frac{3}{2}} \log\left(\frac{1}{\delta}\right) n^{-\frac{1}{2}}\right).
\]
\end{theorem}

\medskip
\noindent

\begin{proof}
Consider the minimizer $\mathbf{c}^*= (c^*_1, c^*_2, \ldots, c^*_l)$, which solves the system $\mathbf{G}_l\mathbf{c} = \mathbf{g}_l$, and the corresponding linear combination $f(\mathbf{c}^*;\cdot)$. 
Then, with probability at least $1-\delta$ we have
\[
\big\|\widetilde{\mathbf{c}}-\mathbf{c}^*\big\|_{\mathbb{R}^l} \leq \|\widetilde{\mathbf{G}}_l^{-1}\|_{op}
\Big( \big\| \tilde{\mathbf{g}}_l - \mathbf{g}_l \big\|_{\mathbb{R}^l} + \big\|\mathbf{G}_l-\widetilde{\mathbf{G}}_l\big\|_{op}\|\mathbf{c}^*\|_{\mathbb{R}^l}\Big)
\leq c\;l \log \left(\frac{1}{\delta} \right)  \; n^{-1/2}.
\]

\noindent
Moreover,
\[
\big\|f^*(\cdot) - f(\widetilde{\mathbf{c}};\cdot)\big\|_{L_2(q_X,Y)} \leq \big\|f^*(\cdot) - f(\mathbf{c}^*;\cdot)\big\|_{L_2(q_X,Y)} + \big\|f(\mathbf{c}^*;\cdot) - f(\widetilde{\mathbf{c}};\cdot)\big\|_{L_2(q_X,Y)},
\]
\[
\big\|f(\mathbf{c}^*;\cdot) - f(\widetilde{\mathbf{c}};\cdot)\big\|_{L_2(q_X,Y)} 
= \Big\|\sum_{j=1}^l (c^*_j-\tilde{c}_j) f_{\mathbf{z},j}\Big\|_{L_2(q_X,Y)}
\leq \sqrt{l}\,\big\|\widetilde{\mathbf{c}}-\mathbf{c}^*\big\|_{\mathbb{R}^l}\,\max_j\|f_{\mathbf{z}, j}\|_{L_2(q_X,Y)}.
\]

The statement of the theorem follows from the definition of $\mathbf{c}^*$ and from the above three inequalities.
\end{proof}

In the next section, we illustrate the performance of the aggregated model $f(\widetilde{\mathbf{c}};\cdot)$ by considering two settings. 

First, we consider the case when the models $f_{\mathbf{z},j}$ are constructed according to (\ref{predictor})  for a fixed scalar multiplicative kernel $k(x,t)$ but for various values of the regularization parameter $\lambda =  \lambda_1,...,\lambda_l$, i.e., $f_{\mathbf{z},j} = \widehat{f}_{\mathbf{z},\lambda_j} $, $j= 1,...,l$. In this setting the model $f(\widetilde{\mathbf{c}};\cdot)$ described in Theorem \ref{Tagg} can be seen as the resolver of the regularization parameter choice issue.

Second, we consider different multiplicative kernels $k(x,t) =k_1(x,t), k_2(x,t),...,k_l(x,t) $ and construct the predictors $\widehat{f}_{\mathbf{z},\lambda_{s} }(k_j; \cdot) $ according to (\ref{predictor}) for various values of the regularization parameter $\lambda =  \lambda_{s}$, $s=1,...,l$, where $\widehat{f}_{\mathbf{z},\lambda }(k_j; \cdot) $ is just a more specific notation for $\widehat{f}_{\mathbf{z},\lambda}(\cdot)$ defined by (\ref{predictor}) with a fixed kernel $k(x,t) =k_j(x,t)$. Then,  for each $k(x,t) =k_j(x,t)$ we aggregate the models $f_{\mathbf{z},s} = \widehat{f}_{\mathbf{z},\lambda_{s} }(k_j; \cdot)$, $s= 1,...,l$,  again in the way described in Theorem \ref{Tagg}, that gives us the set of predictors 
\begin{equation}\label{agglam}
f_{\mathbf{z}}(k_j; \cdot) = f(\widetilde{\mathbf{c}}_j;\cdot)  = \sum_{s=1}^l \tilde{c}^j_s \widehat{f}_{\mathbf{z},\lambda_{s} }(k_j; \cdot), \quad j = 1,...,l,
\end{equation}
and each of them can be seen as the resolver of the regularization parameter choice issue for learning with a fixed kernel $k(x,t) =k_j(x,t)$.

Finally, we use one more time the procedure described in Theorem \ref{Tagg} to aggregate the models $f_{\mathbf{z},j}(\cdot) = f_{\mathbf{z}}(k_j; \cdot)$  associated with different kernels. Then the above procedure becomes an algorithm for multiple kernel learning. From (\ref{eq4.5.4}) and Theorem \ref{Tagg} it follows that the predictor resulted from such multiple kernel learning is expected to perform at the level of the best model among the ones involved in the aggregation. This is illustrated by Table \ref{tab:kernel_performance} below. To the best of our knowledge, the considered algorithm is the first one proposed for multiple kernel learning with functional data.

\section{Numerical Experiments}\label{numerical experiments}

In the section, we follow \cite{aspri} and present numerical experiments conducted on the \textit{Faces} dataset \cite{bainbridge2013intrinsic} 
(The 10k US Adult Faces Database), comprising 10,168 natural images uniformly 
resized to $100 \times 100$ pixels. For the experiments, we first converted 
them to grayscale images. The problem setup is formulated with the objective 
of recovering face images $y \in Y$ from their sinogram representations $x \in X$. 

Non-distorted face images $y_i$ and their sinograms $x_i$ are served as training source data $\mathbf{z} =$ $\{(x_i, y_i)\}_{i=1}^n$, while the sinograms of blurred images, which have been not included in $\mathbf{z}$, or their blurred sinograms, are treated as unlabeled inputs $\{x'_i\}_{i=1}^m$ sampled from the marginal target distribution. They are used for calculating the approximate reweighting function (\ref{beta}) and also as inputs for tested predictors reconstructing the corresponding non-distorted face images considered as ground truth. 

Before proceeding, we recall some necessary definitions.

\begin{definition}[Radon Transform]
For a two-dimensional function $f \in L^1(\mathbb{R}^2)$, the Radon transform 
$R f(\theta, s)$ is defined as
\[
R f(\theta, s) = \int_{-\infty}^{\infty} f(s \cos \theta - t \sin \theta, 
\, s \sin \theta + t \cos \theta) \, dt,
\]
where $\theta \in [0, \pi)$ denotes the projection angle and $s \in \mathbb{R}$ 
denotes the distance from the origin.
\end{definition}

\begin{definition}[Inverse Radon Transform]
Let \( f \in L^1(\mathbb{R}^2) \) and let \( Rf(\theta,s) \) denote its Radon transform.
The inverse Radon transform reconstructs \( f \) from its projections via
\[
f(t,\tau)
=
\int_{0}^{\pi}
\bigl(Rf(\theta,\cdot) * h\bigr)
\bigl(t\cos\theta + \tau\sin\theta\bigr)\, d\theta,
\qquad (t,\tau)\in\mathbb{R}^2,\ \theta\in[0,\pi).
\]
Here \( * \) denotes convolution with respect to the variable \( s \).
The filter \( h \) is defined in the Fourier domain by
\[
(\mathcal{F}h)(\xi)=|\xi|, \qquad \xi\in\mathbb{R},
\]
where \( \mathcal{F}h \) denotes the one-dimensional Fourier transform of \( h \).

\end{definition}

\begin{definition}[Convolution of Two Matrices]
Let $A \in \mathbb{R}^{m \times n}$ be an image matrix and $K \in \mathbb{R}^{p \times q}$ 
a kernel (filter) matrix. The convolution $A * K$ is defined entrywise as
\[
(A * K)(i,j) = \sum_{u=1}^{p} \sum_{v=1}^{q} A(i-u,j-v) \, K(u,v),
\]
where indices outside the range of $A$ are typically handled by zero-padding.  
\end{definition}

\vspace{0.5em}
To convert face images into sinograms (via the Radon transform), we employed 85 projection angles uniformly distributed over the interval $[0^\circ, 180^\circ]$, with number of detectors fixed at 142. Hence the input space $X$ is a subset of  $R^{12,070}$ and output $Y$ space is subset of $R^{10,000}$.
Two experimental setups were considered.  

\vspace{0.5em}
In the first case, each face image is first converted into its corresponding sinogram, 
and a blur kernel is then applied to the sinogram i.e., convolved with the blur kernel matrix (See figure \ref{fig:case1}, \ref{fig:case2}). 
We restricted our study to two types of blur:

\vspace{0.5em}
\begin{itemize}
    \item A \textbf{Gaussian blur kernel} of order $9 \times 9$, with standard deviation
$\sigma = 2$, normalized such that $\sum_{u,v} G(u,v) = 1$. Each entry is defined by
\[
G(u,v)
=
\frac{\exp\!\left(-\frac{u^2+v^2}{2\sigma^2}\right)}
{\sum_{i=-4}^{4}\sum_{j=-4}^{4}
\exp\!\left(-\frac{i^2+j^2}{2\sigma^2}\right)},
\quad
u,v \in \{-4,-3,\dots,4\}.
\]
This corresponds to \emph{out-of-focus blur}.
    \vspace{0.4em}
    \item A \textbf{motion blur kernel} of order $9 \times 9$, consisting entirely of zeros 
    except for the central row, which is set to
    \[
    H = \frac{1}{9}[1,1,1,1,1,1,1,1,1].
    \]
    This corresponds to \emph{horizontal motion blur}.  
\end{itemize}

\vspace{0.5em}
\noindent
These two types of blur were selected since most real-world blurring effects can be 
reasonably seen as combinations of Gaussian (out-of-focus) and motion (directional) blur \cite{blurring}.  
The source dataset consists of clear (non-blurred) sinograms together with their 
corresponding images, whereas the target dataset consists solely of blurred sinograms. 
This setup is treated as a covariate shift adaptation problem, with an equal number 
of source and target samples to avoid imbalance.  

\vspace{0.5em}
In the second case, the blur operation is applied directly to the face images 
prior to generating the sinograms (See figure \ref{fig:case3}, \ref{fig:case4}). This results in blurred sinogram data produced 
through an indirect process. As in the first case, we constructed source and target 
datasets with an equal number of samples, ensuring that the datasets are completely 
disjoint i.e., target samples are not merely blurred versions of the source samples.

\begin{figure}[h!]
    \centering
    \includegraphics[width=0.9\linewidth]{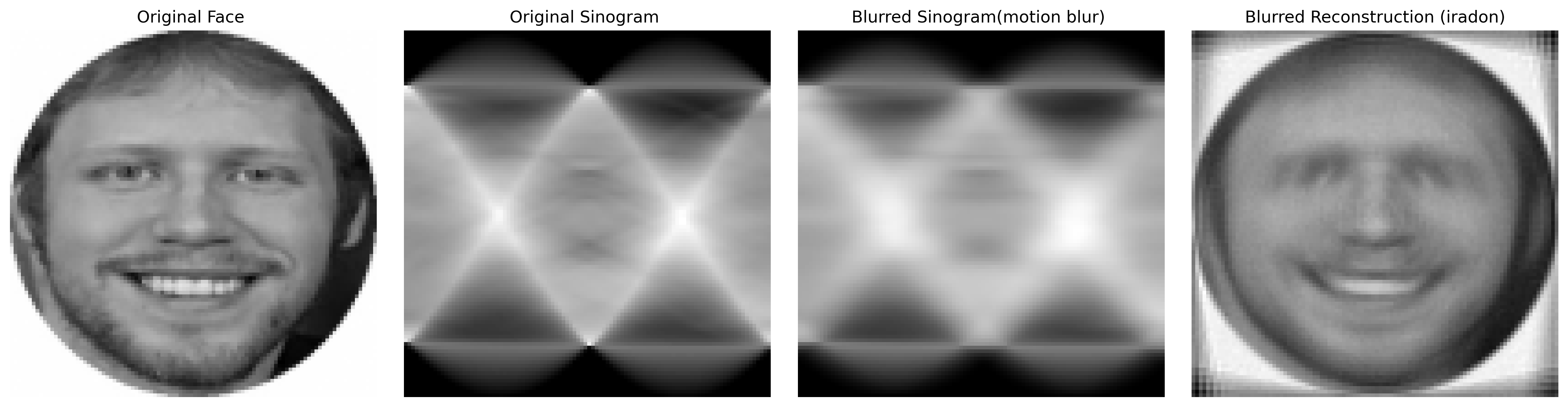}
    \caption{An illustration of Case 1: a face image converted to its sinogram, 
    followed by application of a motion blur kernel.}
    \label{fig:case1}
\end{figure}

\begin{figure}[h!]
    \centering
    \includegraphics[width=0.9\linewidth]{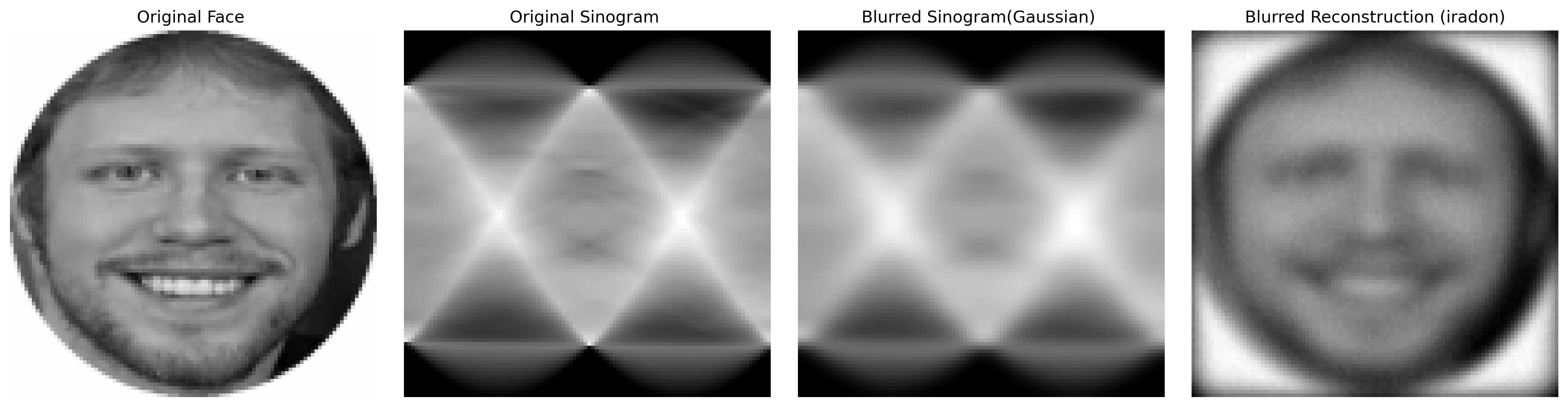}
    \caption{An illustration of Case 1: a face image converted to its sinogram, 
    followed by application of a Gaussian blur kernel.}
    \label{fig:case2}
\end{figure}

\begin{figure}[h!]
    \centering
    \includegraphics[width=0.9\linewidth]{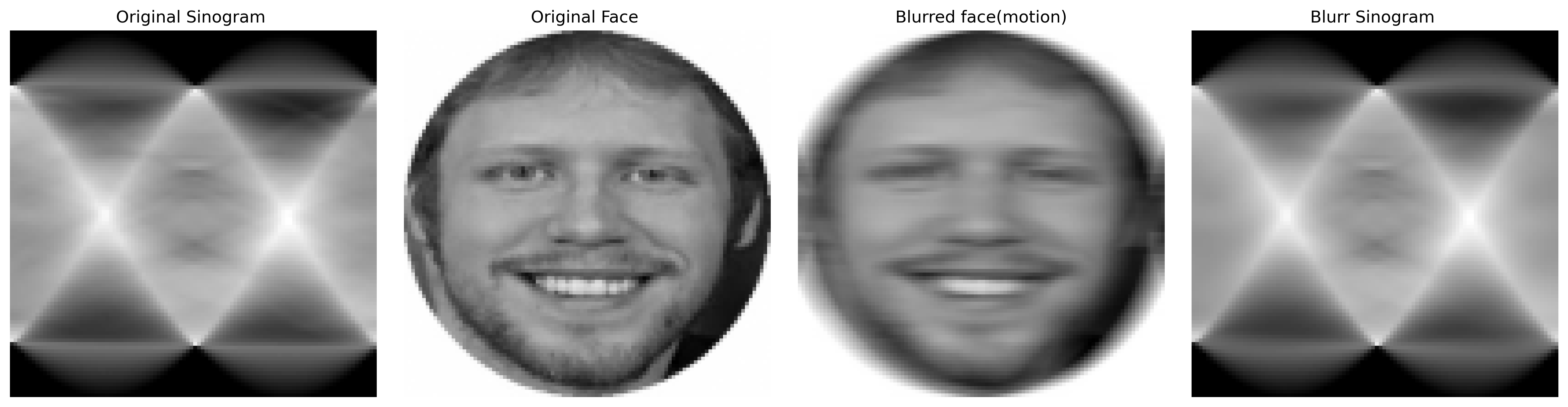}
    \caption{An illustration of Case 2: motion blur applied directly to the face image 
    prior to generating its sinogram.}
    \label{fig:case3}
\end{figure}

\begin{figure}[h!]
    \centering
    \includegraphics[width=0.89\linewidth]{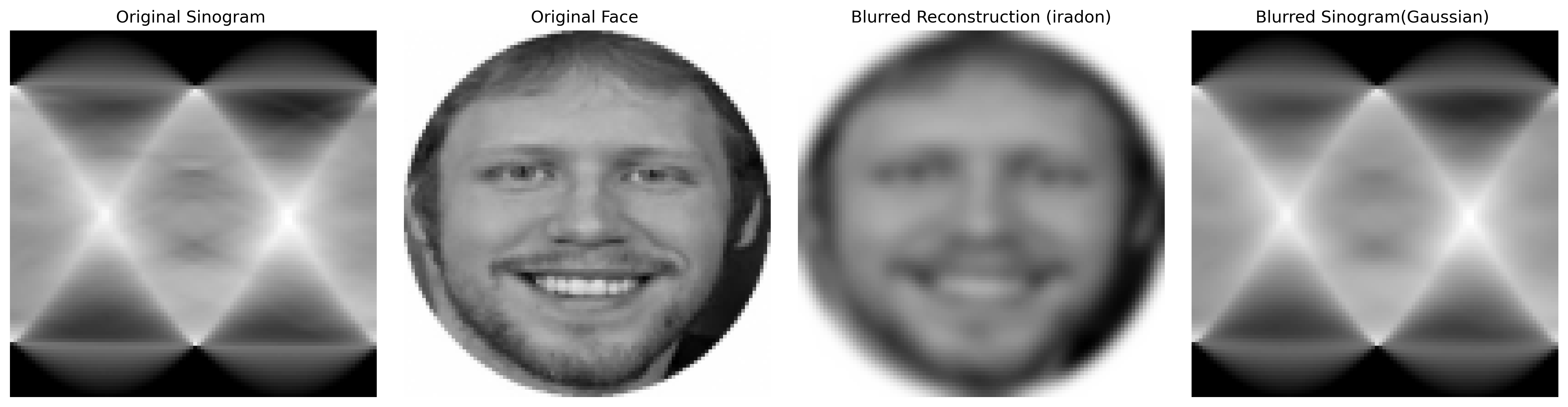}
    \caption{An illustration of Case 2: Gaussian blur applied directly to the face image 
    prior to generating its sinogram.}
    \label{fig:case4}
\end{figure}

\newpage

\vspace{0.5em}
For both cases, we conducted experiments with 1000, 2000, and 3000 samples. Across these experiments, we observed that, despite severe degradation in input data quality due to blurring, the proposed algorithm successfully preserved substantial facial information. 
In all our experiments the error was evaluated using Peak signal-to-noise ratio (PSNR), relative error (Rel. Err.), and the mean squared error (MSE) metrics.
The results of recovery of images $y = \widehat{f}_{\mathbf{z},\lambda}(x)$ by applying (\ref{predictor},  \ref{beta}) with 
\begin{equation}\label{eq:k1}
k(x, \cdot) = \exp\!\left(-\gamma \| x - \cdot\|^2_X\right), \quad \gamma = 10^{-2},
\end{equation}
to blurred sinograms $x$ are shown on Figures \ref{Case 1m},\ref{Case 1g},\ref{Case 1fm},\ref{Case fig7}. For comparison, we include the reconstruction obtained directly using the inverse radon transform (iradon), along with reconstructions from our method for 
n  = 1000, 2000, 3000. The results demonstrate that our algorithm retrieves at least as much, and often more, information than the standard iradon transform.

\newpage

\vspace{0.2em}
\begin{figure}[h!]
    \centering
    \includegraphics[width=0.99\linewidth]{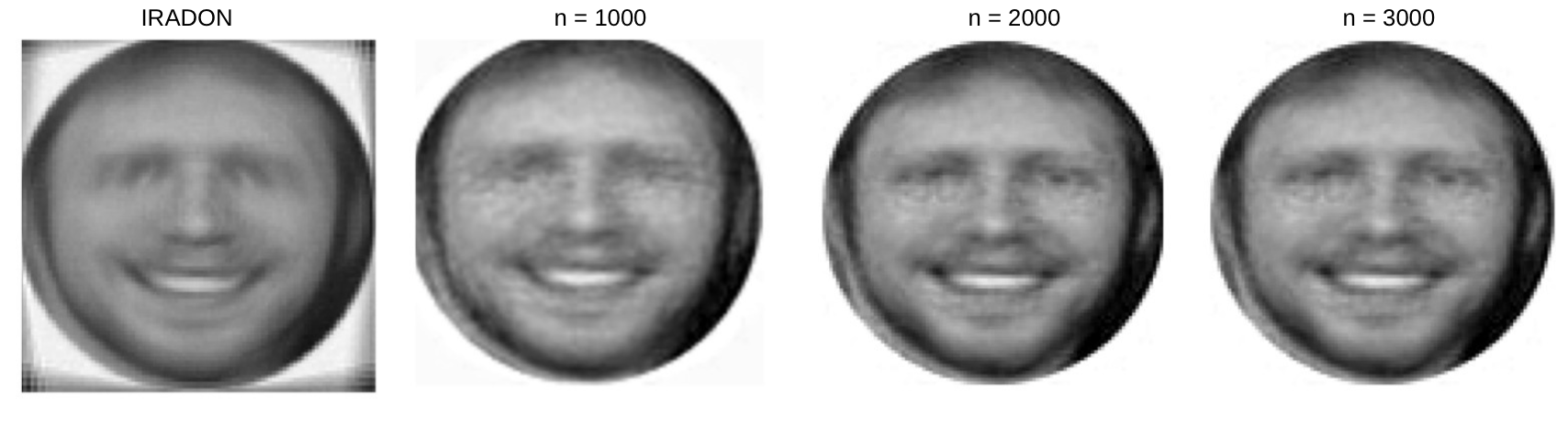}
    \caption{Reconstruction results under motion blur applied to the sinograms.}
    \label{Case 1m}
\end{figure}

\begin{figure}[h!]
    \centering
    \includegraphics[width=01\linewidth]{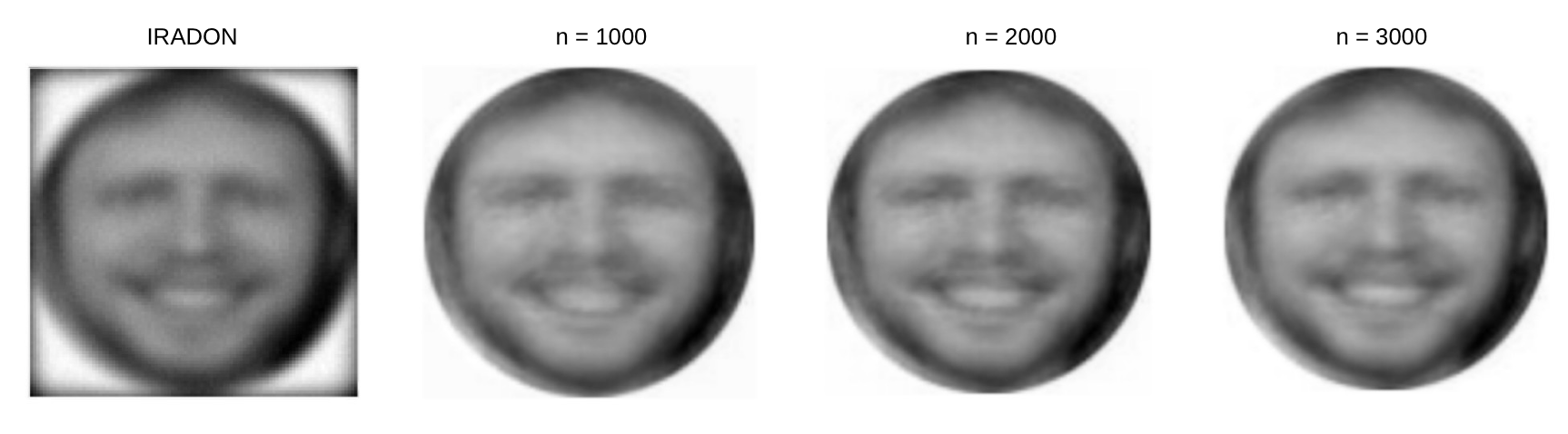}
    \caption{Reconstruction results under Gaussian blur applied to the sinograms.}
    \label{Case 1g}
\end{figure}

\begin{figure}[h!]
    \centering
    \includegraphics[width=1\linewidth]{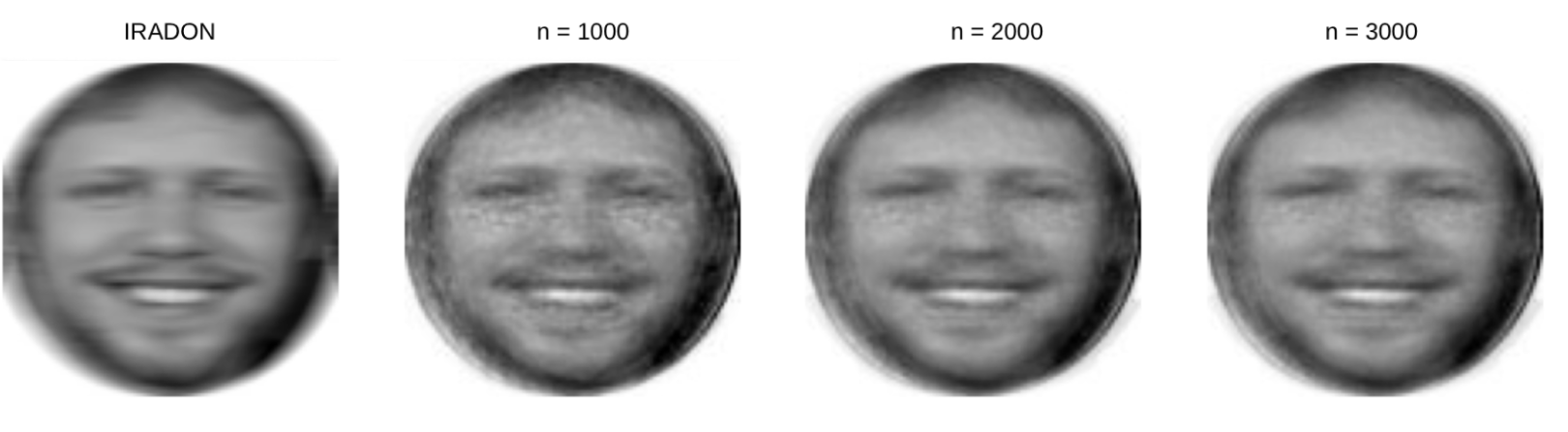}
    \caption{Reconstruction results under motion blur applied to faces.}
    \label{Case 1fm}
\end{figure}

\begin{figure}[h!]
    \centering
    \includegraphics[width=1\linewidth]{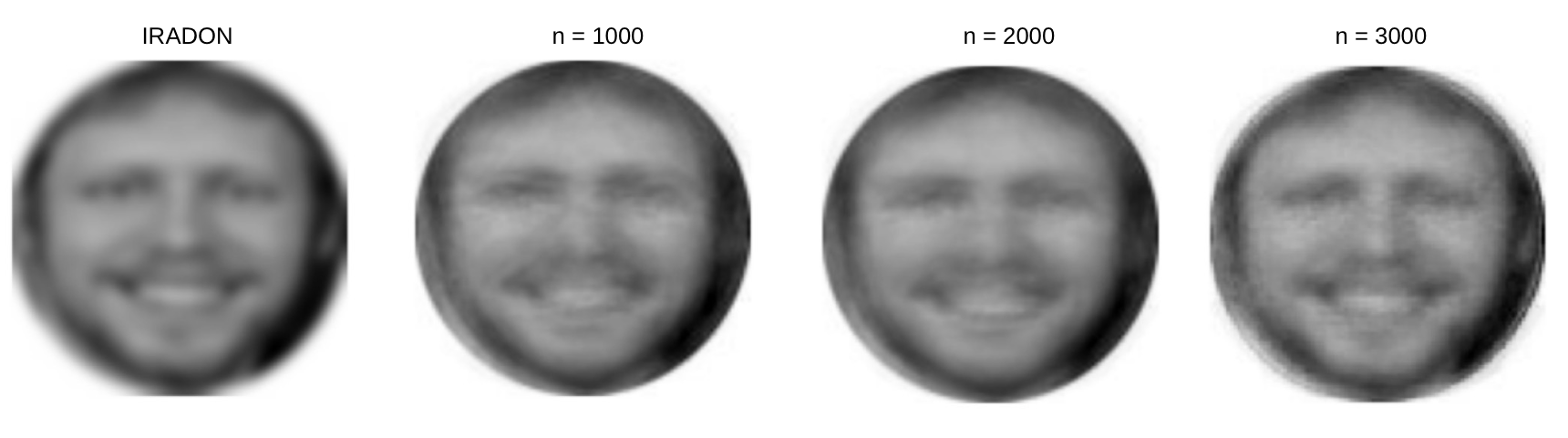}
    \caption{Reconstruction results under Gaussian blur applied to faces.}
    \label{Case fig7}
\end{figure}

\vspace{0.5em}

\newpage

\noindent
In Table~\ref{tab:lambda_sweep_all}, we report the results for Gaussian blur and motion blur with $n = m = 3000$, obtained using regularization parameters $\lambda = \lambda_j = 10^{-j-1}$ for $j = 1,2,3,4,5,6$. As mentioned earlier, we employed the KuLSIF method to compute the weights~$\beta$ (see equation (\ref{beta})). 
A Gaussian kernel (\ref{eq:k1}) was used to estimate the weights and the regularization 
parameter~$\alpha$ was selected according to the quasi-optimality criterion~\cite{ALPHA-SELECTION}.
In addition, the table  shows the performance of the proposed aggregation method (\ref{agg G}, \ref{agg g}) applied to the estimators corresponding to 6 different $\lambda$ values. As shown in the table, the aggregation approach yields results that remain very close to, though slightly below, the best outcomes obtained from the $\lambda$ sweep.

{
\renewcommand{\arraystretch}{0.1}
\begin{table}[t]
\centering
\caption{Lambda sweep results for $n=m=3000$ under different blur settings for the Gaussian kernel.}
\label{tab:lambda_sweep_all}

\begin{minipage}[t]{0.48\textwidth}
\centering
\small
\textbf{(a) Motion blur (sinograms blurred)}\\
\renewcommand{\arraystretch}{1.15}
\begin{tabular}{cccc}
\toprule
Lambda & MSE & Rel. Err. & PSNR \\
\midrule
$10^{-2}$ & 0.027848 & 0.2548 & 15.55 \\
$10^{-3}$ & 0.025360 & 0.2465 & 15.91 \\
$10^{-4}$ & 0.020912 & 0.2258 & 16.80 \\
$10^{-5}$ & 0.013299 & 0.1792 & 18.76 \\
$10^{-6}$ & 0.008150 & 0.1393 & 20.89 \\
$10^{-7}$ & 0.006956 & 0.1290 & 21.58 \\
\midrule
\textbf{Agg.} & \textbf{0.0081} & \textbf{0.1388} & \textbf{20.92} \\
\bottomrule
\end{tabular}
\end{minipage}
\hfill
\begin{minipage}[t]{0.48\textwidth}
\centering
\small
\textbf{(b) Motion blur (faces blurred)}\\
\renewcommand{\arraystretch}{1.15}
\begin{tabular}{cccc}
\toprule
Lambda & MSE & Rel. Err. & PSNR \\
\midrule
$10^{-2}$ & 0.009369 & 0.1453 & 20.28 \\
$10^{-3}$ & 0.005775 & 0.1155 & 22.39 \\
$10^{-4}$ & 0.004328 & 0.1008 & 23.64 \\
$10^{-5}$ & 0.004211 & 0.1003 & 23.76 \\
$10^{-6}$ & 0.004589 & 0.1050 & 23.38 \\
$10^{-7}$ & 0.004827 & 0.1077 & 23.17 \\
\midrule
\textbf{Agg.} & \textbf{0.0049} & \textbf{0.1090} & \textbf{23.10} \\
\bottomrule
\end{tabular}
\end{minipage}

\vspace{0.6em}

\begin{minipage}[t]{0.48\textwidth}
\centering
\small
\textbf{(c) Gaussian blur (sinograms blurred)}\\
\renewcommand{\arraystretch}{1.15}
\begin{tabular}{cccc}
\toprule
Lambda & MSE & Rel. Err. & PSNR \\
\midrule
$10^{-2}$ & 0.017347 & 0.2015 & 17.61 \\
$10^{-3}$ & 0.010324 & 0.1559 & 19.86 \\
$10^{-4}$ & 0.006115 & 0.1201 & 22.14 \\
$10^{-5}$ & 0.004034 & 0.0978 & 23.94 \\
$10^{-6}$ & 0.003751 & 0.0955 & 24.26 \\
$10^{-7}$ & 0.004447 & 0.1047 & 23.52 \\
\midrule
\textbf{Agg.} & \textbf{0.0052} & \textbf{0.1136} & \textbf{22.80} \\
\bottomrule
\end{tabular}
\end{minipage}
\hfill
\begin{minipage}[t]{0.48\textwidth}
\centering
\small
\textbf{(d) Gaussian blur (faces blurred)}\\
\renewcommand{\arraystretch}{1.15}
\begin{tabular}{cccc}
\toprule
Lambda & MSE & Rel. Err. & PSNR \\
\midrule
$10^{-2}$ & 0.01955 & 0.2171 & 17.09 \\
$10^{-3}$ & 0.011911 & 0.1693 & 19.25 \\
$10^{-4}$ & 0.007454 & 0.1336 & 21.27 \\
$10^{-5}$ & 0.005495 & 0.1145 & 22.60 \\
$10^{-6}$ & 0.004742 & 0.1006 & 23.24 \\
$10^{-7}$ & 0.004787 & 0.1029 & 23.20 \\
\midrule
\textbf{Agg.} & \textbf{0.0055} & \textbf{0.1175} & \textbf{22.54} \\
\bottomrule
\end{tabular}
\end{minipage}

\end{table}
}
\newpage
One of the main challenges in kernel-based learning is the selection of an appropriate kernel. To address this, we conducted experiments with four different kernels: the Gaussian kernel $k(x,\cdot) = k_1(x,\cdot)$ given defined in (\ref{eq:k1}), the Cauchy kernel
\[
k(x, \cdot) = k_2(x,\cdot)= \left(1 + \frac{1}{\gamma^2} \| x - \cdot\|^2_X\right)^{-1}, \quad \gamma = 5,
\]
the Exponential kernel
\[
k(x, \cdot) = k_3(x,\cdot)= \exp\!\left(-\frac{1}{\gamma^2} \| x - \cdot\|_X\right), \quad \gamma = 5,
\]
and the inverse multi-quadratic (IMQ) kernel 
\[
k(x, \cdot) = k_4(x,\cdot)= \left(\gamma^2 + \| x - \cdot\|^2_X\right)^{-1/2}, \quad \gamma = 5.
\]

Then we employed our multiple kernel learning approach, which is based on Theorem \ref{Tagg}. Namely, for each of the above kernels, we constructed the aggregated predictor (\ref{agglam}) 
using the regularization parameters $\lambda_s = 10^{-s-1}$, $s = 1,2,3,4,5,6$, and obtained the unified predictor by aggregating them again in the way described at the end of Section \ref{Aggregation}.
 As shown in Table~\ref{tab:kernel_performance}, the aggregated predictor (Agg.) achieves performance close to that of the best individual kernel, indicating its potential to alleviate the risk of selecting an inadequate hypothesis space.

\newpage

\begin{table}[t]
\centering
\caption{Performance results across different kernels under various blur settings for $n=m=3000$.}
\label{tab:kernel_performance}
\small
\renewcommand{\arraystretch}{1.2}

\begin{minipage}[t]{0.47\textwidth}
\centering
\textbf{(a) Motion blur (sinograms blurred)}\\[0.3em]
\begin{tabular}{lccc}
\toprule
Kernel & MSE & Rel.\ Err. & PSNR \\
\midrule
Gaussian     & 0.00810 & 0.13880 & 20.92 \\
Exponential  & 0.00890 & 0.14636 & 20.51 \\
Cauchy       & 0.01060 & 0.15435 & 19.75 \\
IMQ          & 0.00970 & 0.15378 & 20.13 \\
\midrule
\textbf{Agg.} & \textbf{0.00737} & \textbf{0.13124} & \textbf{21.32} \\
\bottomrule
\end{tabular}
\end{minipage}
\hfill
\begin{minipage}[t]{0.47\textwidth}
\centering
\textbf{(b) Motion blur (faces blurred)}\\[0.3em]
\begin{tabular}{lccc}
\toprule
Kernel & MSE & Rel.\ Err. & PSNR \\
\midrule
Gaussian     & 0.00490 & 0.10900 & 23.10 \\
Exponential  & 0.00411 & 0.09965 & 23.82 \\
Cauchy       & 0.00450 & 0.10408 & 23.47 \\
IMQ          & 0.00513 & 0.11150 & 22.89 \\
\midrule
\textbf{Agg.} & \textbf{0.00395} & \textbf{0.09715} & \textbf{24.03} \\
\bottomrule
\end{tabular}
\end{minipage}

\vspace{0.8em}

\begin{minipage}[t]{0.47\textwidth}
\centering
\textbf{(c) Gaussian blur (sinograms blurred)}\\[0.3em]
\begin{tabular}{lccc}
\toprule
Kernel & MSE & Rel.\ Err. & PSNR \\
\midrule
Gaussian     & 0.00524 & 0.11360 & 22.80 \\
Exponential  & 0.00780 & 0.13678 & 21.08 \\
Cauchy       & 0.00820 & 0.14019 & 20.86 \\
IMQ          & 0.00834 & 0.14173 & 20.78 \\
\midrule
\textbf{Agg.} & \textbf{0.00516} & \textbf{0.11206} & \textbf{22.88} \\
\bottomrule
\end{tabular}
\end{minipage}
\hfill
\begin{minipage}[t]{0.47\textwidth}
\centering
\textbf{(d) Gaussian blur (faces blurred)}\\[0.3em]
\begin{tabular}{lccc}
\toprule
Kernel & MSE & Rel.\ Err. & PSNR \\
\midrule
Gaussian     & 0.00550 & 0.11750 & 22.54 \\
Exponential  & 0.00460 & 0.10598 & 23.47 \\
Cauchy       & 0.00450 & 0.10475 & 23.51 \\
IMQ          & 0.00470 & 0.10685 & 23.29 \\
\midrule
\textbf{Agg.} & \textbf{0.00423} & \textbf{0.10265} & \textbf{23.73} \\
\bottomrule
\end{tabular}
\end{minipage}

\end{table}

\newpage
\section{Proofs of Auxiliary Results}\label{Auxillary_results}
\begin{proof}\textbf{of Lemma~\ref{Lemma:3}.}
Let $\{d_k\}_{k\ge 1}$ be an orthonormal basis of $Y$.
Since $C \in \mathcal S_2(\mathcal H,Y)$, it admits the Hilbert–Schmidt expansion
\[
C = \sum_{k,l} a_{k,l}\, d_k \otimes f_l .
\]

\noindent 
Assume that $C \in \mathcal S_2(\mathcal H^\varphi,Y)$, then
$
\|C\|_{\mathcal S_2(\mathcal H^\varphi,Y)}^2 < \infty .
$
Using the definition of the norm on $\mathcal H^\varphi$, we obtain
\begin{align*}
\|C\|_{\mathcal S_2(\mathcal H^\varphi,Y)}^2
&= \Big\| \sum_{k,l} a_{k,l}\, d_k \otimes f_l \Big\|_{\mathcal S_2(\mathcal H^\varphi,Y)}^2 = \sum_{k,l} \frac{a_{k,l}^2}{\varphi^2(\mu_l)} < \infty .
\end{align*}

\noindent
Now define the operator
\[
C_o = \sum_{k,l} \frac{a_{k,l}}{\varphi(\mu_l)} \, d_k \otimes f_l .
\]
By the above summability condition, we have $C_o \in \mathcal S_2(\mathcal H,Y)$. Moreover,
\[
C_o \, \varphi(C_X)
= \sum_{k,l} \frac{a_{k,l}}{\varphi(\mu_l)} \, d_k \otimes \varphi(\mu_l) f_l
= \sum_{k,l} a_{k,l}\, d_k \otimes f_l
= C .
\]
Hence,
$
C = C_o \, \varphi(C_X).
$

\bigskip
Conversely, assume that
\[
C_o = \sum_{k,l} a_{k,l}\, d_k \otimes f_l,
\qquad \text{with} \qquad
\sum_{k,l} a_{k,l}^2 < \infty .
\]
Define $
C := C_o \, \varphi(C_X),$ then we obtain
\[
C = \sum_{k,l} a_{k,l}\, d_k \otimes \varphi(\mu_l)\, f_l .
\]
Hence,
\[
\|C\|_{\mathcal S_2(\mathcal H^\varphi, Y)}^2
= \sum_{k,l} a_{k,l}^2
< \infty,
\]
which implies that \( C \in \mathcal S_2(\mathcal H^\varphi, Y) \).

\end{proof}

\begin{proof}\textbf{of Lemma~\ref{Lemma:1}.}
Let \(\xi\) be an \(\mathcal{S}_2(\mathcal{H})\)-valued random variable defined by
\[
\xi(x)
=
(C_{X}+\lambda I)^{-1/2}\,(\beta(x)\, \phi(x)\otimes \phi(x)) ,\quad x\in X.
\]

\noindent
Then,
\[
\mathbb{E}_p[\xi(x)]
=
(C_{X}+\lambda I)^{-1/2} C_X,
\qquad
\text{and}
\qquad
\|\xi(x)\|_{\mathcal{S}_2(\mathcal{H})}
\le
\frac{b\kappa^2}{\sqrt{\lambda}}.
\]

\noindent
Consider
\begin{align*}
\bigl\|
(C_{X}+\lambda I)^{-1/2}(\phi(x)\otimes \phi(x))
\bigr\|_{\mathcal S_2(\mathcal H)}^2
&=
\sum_{k=1}^\infty
\bigl\|
(C_{X}+\lambda I)^{-1/2}
(\phi(x)\otimes \phi(x)) f_k
\bigr\|_{\mathcal H}^2 \\[0.4em]
&=
\sum_{k=1}^\infty
\bigl\|
(C_{X}+\lambda I)^{-1/2}
\langle \phi(x), f_k\rangle\, \phi(x)
\bigr\|_{\mathcal H}^2 \\[0.4em]
&=
\left(\sum_{k=1}^\infty f_k(x)^2\right)
\bigl\|
(C_{X}+\lambda I)^{-1/2}\phi(x)
\bigr\|_{\mathcal H}^2 \\[0.4em]
&=
k(x,x)\,
\bigl\|
(C_{X}+\lambda I)^{-1/2}\phi(x)
\bigr\|_{\mathcal H}^2 .
\end{align*}

\noindent
Using the expansion \(\phi(x)=\sum_{l=1}^\infty f_l(x) f_l\), we obtain
\begin{align*}
\bigl\|
(C_{X}+\lambda I)^{-1/2}\phi(x)
\bigr\|_{\mathcal H}^2
&=
\sum_{l=1}^\infty
\frac{\langle \phi(x), f_l\rangle^2}{\lambda+\mu_l}
=
\sum_{l=1}^\infty
\frac{f_l(x)^2}{\lambda+\mu_l}.
\end{align*}

\noindent
Consequently,
\begin{align*}
\mathbb{E}_q
\bigl\|
(C_{X}+\lambda I)^{-1/2}(\phi(x)\otimes \phi(x))
\bigr\|_{\mathcal S_2(\mathcal H)}^2
&\leq
\kappa^2
\sum_{l=1}^\infty
\frac{\mu_l}{\lambda+\mu_l}
=
\kappa^2\,\mathcal{N}(\lambda),
\end{align*}
where we used the identity \(\mathbb{E}_q[f_l(x)^2]=\mu_l\). Hence, we have 

\[\mathbb{E}_p\|\xi(x)\|^2_{\mathcal{S}_2(\mathcal{H})}\leq b\kappa^2\mathcal{N}(\lambda).
\]

\noindent
Finally, applying Lemma~\ref{PINELIS:INEQUALITY}, it holds with probability at least \(1-\delta\) that
\[
\bigl\|(C_{X}+\lambda I)^{-1/2} (C_X - \widehat{C}_X^\beta)\bigr\|_{op}
\le
2\left(
\frac{2b\kappa^2}{\sqrt{\lambda}\,n}
+
\sqrt{\frac{b\kappa^2\mathcal{N}(\lambda)}{n}}
\right)
\log\!\left(\frac{2}{\delta}\right).
\]

\end{proof}

\begin{proof}\textbf{of Lemma~\ref{Lemma:2}.}
\noindent
Using the identity \(A^{-1}B = A^{-1}(B-A) + I\) and Lemma~\ref{Lemma:1}, we obtain that, with probability at least \(1-\delta\),
\[
\bigl\|
(C_{X}+\lambda I)^{-1} (\widehat{C}^{\beta}_{X}+\lambda I)
\bigr\|_{op}
\le
\frac{1}{\sqrt{\lambda}}\;
\bigl\|
(C_{X}+\lambda I)^{-1/2}
(  \widehat{C}_X^\beta - C_X)
\bigr\|_{op}
+ 1
\le
\frac{\mathcal{B}_{n,\lambda}\,
\log\!\left(\frac{2}{\delta}\right)}{\sqrt{\lambda}}
+ 1 .
\]

\noindent
Now, using the analysis of Proposition~1 in \cite{Distributed_learning}, we further obtain that, with probability at least \(1-\delta\),
\[
\bigl\|
(C_{X}+\lambda I)
(\widehat{C}^{\beta}_{X}+\lambda I)^{-1}
\bigr\|_{op}
\le
2
\left[
\left(
\frac{\mathcal{B}_{n,\lambda}\,
\log\!\left(\frac{2}{\delta}\right)}{\sqrt{\lambda}}
\right)^2
+ 1
\right].
\]
    
\end{proof}

\begin{proof}\textbf{of Lemma~\ref{lemma:5}.}
By Lemma~\ref{Lemma:1} and assumption \eqref{eq:condition_on_lambda}, we obtain with probability at least $1-\delta$ that
\begin{equation}\label{eq:lemma5-core}
\bigl\|
(\widehat{C}^\beta_X-C_X)(C_X+\lambda I)^{-1}
\bigr\|_{{op}}
\;\le\;
\frac12 .
\end{equation}

\noindent
Next, observe that
\[
(C_X+\lambda I)(\widehat{C}^\beta_X+\lambda I)^{-1}
=
\Bigl[
I+(\widehat{C}^\beta_X-C_X)(C_X+\lambda I)^{-1}
\Bigr]^{-1}.
\]
Using \eqref{eq:lemma5-core} together with the Neumann series argument, we deduce
\[
\bigl\|
(C_X+\lambda I)(\widehat{C}^\beta_X+\lambda I)^{-1}
\bigr\|_{{op}}
\;\le\;
\frac{1}{1-
\bigl\|(\widehat{C}^\beta_X-C_X)(C_X+\lambda I)^{-1}\bigr\|_{op}}
\;\le\;
2.
\]
Finally, since $0<l<1$, we apply Cordes inequality \cite{Cordes_1987} to get
\[
\bigl\|
(C_X+\lambda I)^l(\widehat{C}^\beta_X+\lambda I)^{-l}
\bigr\|_{{op}}
\le
2^{\,l}.
\]
This completes the proof.

\end{proof}

\begin{proof}\textbf{of Lemma~\ref{L_2(lemma-A)}.}
The following proof is obtained by adapting the analysis of Theorem 1 in \cite{Gupta_2025} to our setting. We start with
\begin{align*}
\big\|(C_*-\bar{C}_{\lambda}^\beta)\, C_X^{1/2} \big\|_{\mathcal S_2(\mathcal H,Y)}
\leq
\bigl\|
(C_X+\lambda I)^{1/2}
r_\lambda(\widehat C_{X}^{\beta}) (C_*)^*
\bigr\|_{\mathcal S_2(Y,\mathcal H)} .
\end{align*}

\noindent
By Lemma~\ref{lemma:5} (with $l=\tfrac12$) and equation \eqref{eq:condition_on_lambda}, we have 
\[
\big\|(C_X+\lambda I)^{1/2}(\widehat{C}^{\beta}_{X}+\lambda I)^{-1/2}\big\|_{op}
\le \sqrt{2}.
\]
Hence,
\begin{align*}
\big\|(C_*-\bar{C}_{\lambda}^\beta)\, C_X^{1/2} \big\|_{\mathcal S_2(\mathcal H,Y)}
&\le \sqrt{2} \,
\bigl\|
(\widehat{C}^{\beta}_{X}+\lambda I)^{1/2}
r_\lambda(\widehat C_X^{\beta})
(C_*)^*
\bigr\|_{\mathcal S_2(Y,\mathcal H)} .
\end{align*}

\noindent
Insert the identity
\[
I
=
(\widehat C_X^{\beta}+\lambda I)^{\nu-\frac12}
(\widehat C_X^{\beta}+\lambda I)^{-(\nu-\frac12)},
\]
to obtain
\begin{align*}
&\big\|(C_*-\bar{C}_{\lambda}^\beta)\, C_X^{1/2} \big\|_{\mathcal S_2(\mathcal H,Y)} \le \sqrt{2} \,
\bigl\|
(\widehat{C}^{\beta}_{X}+\lambda I)^{1/2}
r_\lambda(\widehat C_X^{\beta})
(\widehat C_X^{\beta}+\lambda I)^{\nu-\frac12}\bigr\|_{op}\\
&\hspace{12em} \times \bigl\|(\widehat C_X^{\beta}+\lambda I)^{-(\nu-\frac12)}
(C_*)^*
\bigr\|_{\mathcal S_2(Y,\mathcal H)} .
\end{align*}

\noindent
Now observe that
\begin{align*}
\bigl\|
(\widehat{C}_{X}^{\beta} +\lambda I)^{1/2}
r_\lambda(\widehat C_X^{\beta})
(\widehat C_X^{\beta}+\lambda I)^{\nu-\frac12}
\bigr\|_{{op}}
&=
\bigl\|
r_\lambda(\widehat C_X^{\beta})
(\widehat C_X^{\beta}+\lambda I)^\nu
\bigr\|_{{op}}  \le
2^\nu \lambda^\nu(\gamma_\nu+\gamma).
\end{align*}

\noindent
Therefore,
\begin{align*}
\big\|(C_*-\bar{C}_{\lambda}^\beta)\, C_X^{1/2} \big\|_{\mathcal S_2(\mathcal H,Y)}
&\le
2^{\nu+\frac12}\lambda^\nu(\gamma_\nu+\gamma)
\bigl\|
(\widehat C_X^{\beta}+\lambda I)^{-(\nu-\frac12)}
(C_*)^*
\bigr\|_{\mathcal S_2(Y,\mathcal H)} .
\end{align*}

\noindent
Let $\nu' = \nu-\tfrac12$. Then
\begin{align*}
\big\|(C_*-\bar{C}_{\lambda}^\beta)\, C_X^{1/2} \big\|_{\mathcal S_2(\mathcal H,Y)}
&\le
2\cdot 2^{\nu'}\lambda^{\nu'+\frac12}(\gamma_\nu+\gamma)
\bigl\|
(\widehat C_X^{\beta}+\lambda I)^{-\nu'}
(C_*)^*\bigr\|_{\mathcal S_2(Y,\mathcal H)}.
\end{align*}

\noindent
Let $\bar{\nu}$ denotes integer less than $\nu'$, then applying Lemma~\ref{lemma:5} again yields
\begin{align*}
\big\|(C_*-\bar{C}_{\lambda}^\beta)\, C_X^{1/2} \big\|_{\mathcal S_2(\mathcal H,Y)}
&\le 
2^{\nu'+1}
\bigl[\sqrt \lambda(\gamma_\nu+\gamma)\bigr]
\lambda^{\nu'} \Bigl\|
(\widehat C_X^{\beta}+\lambda I)^{-(\nu'-\bar \nu)}
(C_X+\lambda I)^{(\nu'-\bar \nu)} 
\bigr\|_{op} \\
& \qquad \times
\Bigl\|
(C_X+\lambda I)^{-(\nu'-\bar \nu)}
(\widehat C_X^{\beta}+\lambda I)^{-\bar \nu}
(C_*)^*
\Bigr\|_{\mathcal S_2(Y,\mathcal H)} \\
&\!\!\!\!\!\!\!\!\!\!\!\le
2^{2\nu-\bar \nu}
\bigl[\sqrt \lambda(\gamma_\nu+\gamma)\bigr]
\lambda^{\nu'}
\Bigl\|
(C_X+\lambda I)^{-(\nu'-\bar \nu)}
(\widehat C_X^{\beta}+\lambda I)^{-\bar \nu}
(C_*)^*
\Bigr\|_{\mathcal S_2(Y,\mathcal H)}.
\end{align*}

\noindent
Applying the triangle inequality, we obtain
\begin{align*}
&\lambda^{\nu'}
\bigl\|
(C_X+\lambda I)^{-(\nu'-\bar \nu)}
(\widehat C_X^{\beta}+\lambda I)^{-\bar \nu}
(C_*)^*
\bigr\|_{\mathcal S_2(Y,\mathcal H)} \\
&=
\lambda^{\nu'}
\bigl\|
\bigl[
(C_X+\lambda I)^{-(\nu'-\bar \nu)}
(\widehat C_X^{\beta}+\lambda I)^{-\bar \nu}
-
(C_X+\lambda I)^{- \nu'}
+
(C_X+\lambda I)^{- \nu'}
\bigr]
(C_*)^*
\bigr\|_{\mathcal S_2(Y,\mathcal H)} \\
&\le
\lambda^{\bar \nu} \underbrace{
\bigl\| [
(\widehat C_X^{\beta}+\lambda I)^{-\bar \nu}
-
(C_X+\lambda I)^{-\bar \nu}]
(C_*)^*
\bigr\|_{\mathcal S_2(Y,\mathcal H)}
}_{I}
+
\underbrace{
\lambda^{\nu'}
\|(C_X+\lambda I)^{-\nu'}(C_*)^*\|_{\mathcal S_2(Y,\mathcal H)}
}_{II}.
\end{align*}

\noindent
We now analyze \emph{Term~I} as follows.

\begin{align*}
I
&:= \Bigl\| \bigl( (\widehat{C}^\beta_{X}+\lambda I)^{-\bar \nu}
      - (C_X+\lambda I)^{-\bar \nu} \bigr)
      \varphi(C_X)C_o^{*} \Bigr\|_{\mathcal S_2(Y,\mathcal H)}
      \nonumber\\[0.4em]
&\le
\Bigl\| (\widehat{C}^\beta_{X}+\lambda I)^{-(\bar \nu-1)}
       \bigl( (\widehat{C}^\beta_{X}+\lambda I)^{-1}
       - (C_X+\lambda I)^{-1} \bigr)
       \varphi(C_X)C_o^{*}
\Bigr\|_{\mathcal S_2(Y,\mathcal H)}
\nonumber\\
&\quad
+ \Biggl\|
\sum_{i=1}^{\bar \nu-1}
(\widehat{C}^\beta_{X}+\lambda I)^{-i}
(C_X-\widehat{C}^\beta_{X})
(C_X+\lambda I)^{-(\bar \nu+1-i)}
\varphi(C_X)C_o^{*}
\Biggr\|_{\mathcal S_2(Y,\mathcal H)}
\nonumber\\[0.5em]
&\le
\frac{1}{\lambda^{\bar \nu}}
\Bigl\|
(\widehat{C}^\beta_{X}+\lambda I)\bigl( (\widehat{C}^\beta_{X}+\lambda I)^{-1}
       - (C_X+\lambda I)^{-1} \bigr)\Bigr\|_{{op}}
\,
\|(C_{*})^*\|_{\mathcal S_2(Y,\mathcal H)}
\nonumber\\
&\quad
+ \sum_{i=1}^{\bar \nu-1}
\frac{1}{\lambda^{\,i+\bar \nu-i}}
\Bigl\|
(C_X-\widehat{C}^\beta_{X})(C_X+\lambda I)^{-1}
\Bigr\|_{{op}}
\,
\|(C_{*})^*\|_{\mathcal S_2(Y,\mathcal H)}
\nonumber\\[0.4em]
&=
\frac{1}{\lambda^{\bar \nu}}
\Bigl\|
(C_X-\widehat{C}^\beta_{X})(C_X+\lambda I)^{-1}
\Bigr\|_{{op}}
\,
\|C_{*}\|_{\mathcal S_2(\mathcal H,Y)}
\nonumber
\\
&\quad
+ \sum_{i=1}^{\bar \nu-1}
\frac{1}{\lambda^{\bar \nu}}
\Bigl\|
(C_X-\widehat{C}^\beta_{X})(C_X+\lambda I)^{-1}
\Bigr\|_{{op}}
\,
\|C_{*}\|_{\mathcal S_2(\mathcal H,Y)}
\nonumber\\[0.4em]
&=
\frac{\bar \nu}{\lambda^{\bar \nu}}
\Bigl\|
(C_X-\widehat{C}^\beta_{X})(C_X+\lambda I)^{-1}
\Bigr\|_{{op}}
\,
\|C_{*}\|_{\mathcal S_2(\mathcal H,Y)}.
\end{align*}

\noindent
Note that the first inequality is obtained using Lemma \ref{Naveen-acha:lemma}.

\vspace{0.4em}

\noindent
Using Lemma~\ref{Lemma:1}, with probability at least $1-\delta$,
\begin{equation}
I
\le
\frac{2\bar \nu}{\lambda^{\bar \nu}}
\left(
\frac{2b\kappa^2}{\lambda n}
+
\sqrt{\frac{b\kappa^2\mathcal N(\lambda)}{n\lambda}}
\right)
\log\!\left(\frac{2}{\delta}\right)
\|C_{*}\|_{\mathcal S_2(\mathcal H,Y)} .
\end{equation}

\noindent
To obtain a bound for \emph{Term~II}, we observe that
\begin{align*}
\lambda^{\nu'}
\bigl\|(C_X+\lambda I)^{-\nu'}\varphi(C_X)C_o^*\bigr\|_{\mathcal S_2(Y,\mathcal H)}
&\le
\lambda^{\nu'}
\bigl\|(C_X+\lambda I)^{-\nu'}\varphi(C_X)\bigr\|_{{op}}
\,\|C_o\|_{\mathcal S_2(\mathcal H,Y)}.
\nonumber\\
\end{align*}

\noindent
Now, using Lemma~2 in \cite{Gupta_2025}, we have

\[\lambda^{\nu'}
\bigl\|(C_X+\lambda I)^{-\nu'}\varphi(C_X)C_o^*\bigr\|_{\mathcal S_2(Y,\mathcal H)}\le
\max\!\left(1,\frac{1}{c}\right) 
\varphi(\lambda)\|C_o\|_{\mathcal S_2(\mathcal H,Y)}.\]

\noindent
Combining the bounds for Terms~I and~II, with probability at least $1-\delta$, we have
\begin{align*}
&\bigl\|(C_*-\bar{C}_{\lambda}^\beta) C_X^{1/2}
\bigr\|_{\mathcal S_2(\mathcal H,Y)}
\le
2^{\,2\nu-\bar \nu}
(\gamma_\nu+\gamma)\max\!\left(1,\frac{1}{c}\right) \\
&\quad\times
\Biggl[
2\bar \nu\,\|C_{*}\|_{\mathcal S_2(\mathcal H,Y)}
\left(
\frac{2b\kappa^2}{\sqrt \lambda n}
+
\sqrt{\frac{b\kappa^2\mathcal N(\lambda)}{n}}
\right)
\log\!\left(\frac{2}{\delta}\right) 
+
\varphi(\lambda)\sqrt{\lambda}\,
\|C_o\|_{\mathcal S_2(\mathcal H,Y)}
\Biggr].
\end{align*}

\noindent
The desired result follows, which completes the proof.
\end{proof}

\begin{proof}\textbf{of Lemma~\ref{L2:bound-B}.}
We begin by observing that
\begin{align*}
\bigl\| ( \widehat{C}_\lambda^\beta-\bar{C}_{\lambda}^\beta)\, C_X^{1/2} \bigr\|_{\mathcal  S_2(\mathcal H,Y)}
\leq
\bigl\|
( \widehat{C}_{Y,X}^\beta-{C_*}\widehat C_X^\beta)\,
g_\lambda(\widehat{C}_X^\beta)\,
(C_{X}+\lambda I)^{1/2}
\bigr\|_{\mathcal S_2(\mathcal H,Y)} .
\end{align*}

\noindent
Next, we factorize the operator as
\begin{align*}
&\bigl\|
(\widehat{C}_{Y,X}^\beta - C_* \widehat{C}_X^\beta)\,
g_\lambda(\widehat{C}_X^\beta)\,
(C_X+\lambda I)^{1/2}
\bigr\|_{\mathcal S_2(\mathcal H,Y)} \\
&=
\Bigl\|
\bigl(\widehat{C}_{Y,X}^\beta - C_* \widehat{C}_X^\beta\bigr)
(C_X+\lambda I)^{-1/2}
\;
(C_X+\lambda I)^{1/2}
(\widehat{C}_X^\beta+\lambda I)^{-1/2} \\
&\qquad\qquad
(\widehat{C}_X^\beta+\lambda I)^{1/2}
g_\lambda(\widehat{C}_X^\beta)
(\widehat{C}_X^\beta+\lambda I)^{1/2}
\;
(\widehat{C}_X^\beta+\lambda I)^{-1/2}
(C_X+\lambda I)^{1/2}
\Bigr\|_{\mathcal S_2(\mathcal H,Y)} \\
&\le
\underbrace{\bigl\|
(\widehat{C}_{Y,X}^\beta - C_* \widehat{C}_X^\beta)
(C_X+\lambda I)^{-1/2}
\bigr\|_{\mathcal S_2(\mathcal H,Y)}}_{A}
\;\\
& \hspace{9.5em}  \times 
\underbrace{\bigl\|
(C_X+\lambda I)^{1/2}
(\widehat{C}_X^\beta+\lambda I)^{-1/2}
\bigr\|_{op}^2}_{B}
\; 
\underbrace{\bigl\|
(\widehat{C}_X^\beta+\lambda I)
g_\lambda(\widehat{C}_X^\beta)
\bigr\|_{op}}_{C}.
\end{align*}

\noindent
To bound term $A$, we define the random variable
\[
\xi : X \times Y \to \mathcal S_2(\mathcal H,Y), 
\qquad
\xi(x,y)
=
\beta(x)\,
\bigl(y-C_*\phi(x)\bigr)
\otimes
(C_X+\lambda I)^{-1/2}\phi(x).
\]
Let $M=\|y\|_Y+\|C_*\|_{\mathcal S_2(\mathcal H,Y)}\kappa$, then
\[
\|\xi(x,y)\|_{\mathcal S_2(\mathcal H,Y)}
\leq
b\,
\frac{\|y\|_Y+\|C_*\|_{\mathcal S_2(\mathcal H,Y)}\kappa}{\sqrt{\lambda}}
\,\kappa =
\frac{b\,M\,\kappa}{\sqrt{\lambda}}
 .
\]

\noindent
Moreover,
\begin{align*}
\mathbb{E}_p\bigl[\|\xi\|_{\mathcal S_2(\mathcal H,Y)}^2\bigr]
&\leq
(\|y\|_Y+\|C_*\|_{\mathcal S_2(\mathcal H,Y)}\kappa)^2
\int
\beta^2(x)
\bigl\|
(C_X+\lambda I)^{-1/2}\phi(x)
\bigr\|_{\mathcal H}^2
\,dp_X(x) \\
&\leq
b(\|y\|_Y+\|C_*\|_{\mathcal S_2(\mathcal H,Y)}\kappa)^2
\int
\bigl\|
(C_X+\lambda I)^{-1/2}\phi(x)
\bigr\|_{\mathcal H}^2
\,dq_X(x).
\end{align*}

\noindent
Using Lemma~17 of \cite{meunier2024optimal}, we obtain
\[
\mathbb{E}\bigl[\|\xi\|_{\mathcal S_2(\mathcal H,Y)}^2\bigr]
\leq
bM^2
\mathcal{N}(\lambda).
\]

\noindent
Applying Lemma~\ref{PINELIS:INEQUALITY}, we obtain that with probability at least $1-\delta$,
\begin{align*}
\bigl\|
( \widehat{C}_{Y,X}^\beta-{C_*}\widehat C_X^\beta)\,
(C_X+\lambda I)^{-1/2}
\bigr\|_{\mathcal  S_2(\mathcal H,Y)}
\;\leq\;
\frac{4bM\kappa}{n\sqrt{\lambda}}
\log\!\left( \frac{2}{\delta} \right)
+ 2M
\sqrt{
\frac{b\mathcal{N}(\lambda)}{n}
}\log\!\left( \frac{2}{\delta} \right).
\end{align*}

\noindent
Applying Cordes inequality \cite{Cordes_1987}, Lemma~\ref{Lemma:2} and equation \eqref{eq:condition_on_lambda} yields the bound for term $B$:
\begin{align*}
\bigl\|
(C_{X}+\lambda I)^{1/2}
(\widehat{C}^{\beta}_{X}+\lambda I)^{-1/2}
\bigr\|_{op}^2
\leq
2\Biggl[
\left(
\frac{\mathcal{B}_{n,\lambda}
\log\!\left(\frac{2}{\delta}\right)}{\sqrt{\lambda}}
\right)^2
+1
\Biggr] \leq \frac{5}{2}.
\end{align*}

\noindent
By the defining properties of the regularization family, term $C$ satisfies
\[
\bigl\|
(\widehat{C}^{\beta}_{X}+\lambda I)
g_\lambda(\widehat{C}_X^\beta)
\bigr\|_{op}
\leq
B+D.
\]

Combining the three bounds yields the desired result, completing the proof.
\end{proof}

\section{ACKNOWLEDGMENT}

Vaibhav Silmana thanks the University Grants Commission (UGC), Government of India, for its ﬁnancial assistance (Student ID. 221610001628).

\vskip 0.2in
\bibliography{references}

\end{document}